\documentclass[10pt]{amsart}
\usepackage{amsmath,amsthm,amssymb,url}
\include{xy}
\xyoption{all}
\SelectTips{cm}{}

\newcommand\Z{{\mathbb Z}}
\newcommand\Zpos{\Z_{\ge1}}

\newcommand\Znn{\Z_{\ge0}}
\newcommand\Q{{\mathbb Q}}
\newcommand\Qx{\Q^\times}
\newcommand\Qnn{\Q_{\ge0}}
\newcommand\R{{\mathbb R}}

\newcommand\Rnn{\R_{\ge0}}
\newcommand\C{{\mathbb C}}

\newcommand\lset{\{\,}
\newcommand\rset{\,\}}
\newcommand\lra{\longrightarrow}
\newcommand\bs{\backslash}
\newcommand\inv{^{-1}}
\newcommand\tr{\operatorname{tr}}

\newcommand\eg{{\rm e.g.}}
\newcommand\ie{{\rm i.e.}}
\newcommand\cf{{\rm cf.}}

\newcommand\ip[2]{\langle #1,#2\rangle} 

\newcommand\siX[1]{{\mathcal X}_{#1}} 
\newcommand\Xtwo{\siX2}

\newcommand\XtwoN{\siX2(N)}
\newcommand\XtwoNsemi{\Xtwo^{\rm semi}(N)}

\newcommand\SLgp{\operatorname{SL}}
\newcommand\SL[2]{\SLgp_{#1}(#2)}
\newcommand\SLtwoZ{\SL2\Z}
\newcommand\SLtwoQ{\SL2\Q}

\newcommand\GLgp{\operatorname{GL}}

\newcommand\GL[2]{\GLgp_{#1}(#2)}
\newcommand\GLtwoZ{\GL2\Z}
\newcommand\GLtwoQ{\GL2\Q}
\newcommand\GLtwoR{\GL2\R}

\newcommand\Spgp{\operatorname{Sp}}
\newcommand\Sp[2]{\operatorname{Sp}_{#1}(#2)}
\newcommand\SptwoZ{\Sp2\Z}
\newcommand\SptwoQ{\Sp2\Q}
\newcommand\SptwoR{\Sp2\R}
\newcommand\Gamzero{\Gamma_{\!0}}
\newcommand\GzN{\Gamzero(N)}

\newcommand\Ptwoone{\operatorname{P}_{2,1}}

\newcommand\PtwooneQ{\Ptwoone(\Q)}
\newcommand\PtwooneZ{\Ptwoone(\Z)}

\newcommand\rmat[4]{\left[\begin{array}{rr}
{#1}&{#2}\\{#3}&{#4}\end{array}\right]}
\newcommand\cmat[4]{\left[\begin{array}{cc}
{#1}&{#2}\\{#3}&{#4}\end{array}\right]}
\newcommand\smallmat[4]{\left[\begin{smallmatrix}
{#1}&{#2}\\{#3}&{#4}\end{smallmatrix}\right]}

\newcommand\smallmatabcd{\smallmat abcd}

\newcommand\matoplus{\boxplus}

\newcommand\paramodulargroup[1]{{\rm K}(#1)}
\newcommand\KN{\paramodulargroup N}

\newcommand\pvar{{\cmat\tau zz\omega}} 
\newcommand\smallpvar{{\smallmat\tau zz\omega}}

\newcommand\smallpindN{{\smallmat n{r/2}{r/2}{mN}}}
\newcommand\smallpindnN{{\smallmat m{r/2}{r/2}{nN}}}
\newcommand\fcJ[2]{c(#1;#2)} 

\newcommand\GampmsupzeroN{\Gamma_\pm^0(N)}

\newcommand\UHP{{\mathcal H}} 
\newcommand\UHPtwo{{\UHP_2}}

\newcommand\Matsplus[2]{{\rm M}^+_{#1}(#2)}
\newcommand\MatsplustwoZ{\Matsplus2\Z}

\newcommand\wtvar[2]{[#1]_{#2}}
\newcommand\wtk[1]{\wtvar{#1}k}

\newcommand\fc[2]{a(#1;#2)} 
\newcommand\e{{\rm e}} 

\newcommand\MFsnoweight{{\mathcal M}} 
\newcommand\MFswtgp[2]{\MFsnoweight_{#1}(#2)}
\newcommand\MFs[1]{\MFswtgp k{#1}} 
\newcommand\CFsnoweight{{\mathcal S}} 
\newcommand\CFswtgp[2]{\CFsnoweight_{#1}(#2)}
\newcommand\CFs[1]{\CFswtgp k{#1}} 

\newcommand\JFind{m} 
\newcommand\BPc{c} 
\newcommand\BPA{A} 

\newcommand\MkKN{\MFswtgp k\KN}

\newcommand\MwtKlevel[2]{\MFswtgp{#1}{\paramodulargroup{#2}}}
\newcommand\SkKN{\CFswtgp k\KN}
\newcommand\StwoKN{\CFswtgp2\KN}
\newcommand\StwoKlevel[1]{\CFswtgp2{\paramodulargroup{#1}}}
\newcommand\SwtKlevel[2]{\CFswtgp{#1}{\paramodulargroup{#2}}}

\newcommand\Jac[2]{{\rm J}_{#1,#2}}
\newcommand\Jzerom{\Jac0m}
\newcommand\Jackm{\Jac km}

\newcommand\Jcusp[2]{{\rm J}_{#1,#2}^{\rm cusp}}
\newcommand\Jweak[2]{{\rm J}_{#1,#2}^{\rm weak}}

\newcommand\Jkmcusp{\Jcusp km}
\newcommand\JkNcusp{\Jcusp kN}

\newcommand\JkmNcusp{\Jcusp k{mN}}
\newcommand\JkBPcNcusp{\Jcusp k{\BPc N}}

\newcommand\Jwh[2]{{\rm J}_{#1,#2}^{\rm w.h.}}

\newcommand\JzeroNwh{\Jwh0N}

\newcommand\Grit{{\rm Grit}}

\newcommand\mymod{\text{ mod }}

\renewcommand{\Im}[1]{{\rm Im}(#1)}

\newcommand\TB{{\rm TB}}
\newcommand\BL{{\rm Borch}}

\newcommand\JFtoSMF[2]{E_{#2}{#1}}
\newcommand\JFtoSMFm[1]{\JFtoSMF{#1}m}

\newcommand\nmin{{n_{\rm min}}}
\newcommand\nmax{{n_{\rm max}}}
\newcommand\nextra{{n_{\rm extra}}}

\newcommand\LaurentZ{\Z[\zeta,\zeta\inv]}
\newcommand\LaurentQ{\Q[\zeta,\zeta\inv]}
\newcommand\LaurentC{\C[\zeta,\zeta\inv]}
\newcommand\bigOh[1]{{\mathcal O}(#1)}

\newcommand\ord{\operatorname{ord}}
\newcommand\divv{\operatorname{div}}
\newcommand\Adj{\operatorname{Adj}}

\newcommand\atom[1]{\vartheta^*_{#1}}

\newcommand\iotaone{\operatorname{\iota_1}}
\newcommand\iotaonesigma{\iota_1\sigma}
\newcommand\iotaoneAdjsigma{\iota_1{\rm Adj}(\sigma)}
\newcommand\noway{\rho}
\newcommand\ftss{f_{277}}

\theoremstyle{plain}
\newtheorem{theorem}{Theorem}[section]
\newtheorem{lemma}[theorem]{Lemma}
\newtheorem{proposition}[theorem]{Proposition}
\newtheorem{corollary}[theorem]{Corollary}

\theoremstyle{definition}
\newtheorem{definition}[theorem]{Definition}

\numberwithin{equation}{section}

\begin{document}
\title[Finding all Borcherds products]
{Finding all Borcherds product paramodular cusp forms of a given
weight and level}

\author[C.~Poor]{Cris Poor}
\address{Department of Mathematics, Fordham University, Bronx, NY 10458 USA}
\email{poor@fordham.edu}

\author[J.~Shurman]{Jerry Shurman}
\address{Department of Mathematics, Reed College, Portland, OR 97202 USA}
\email{jerry@reed.edu}

\author[D.~Yuen]{David S.~Yuen}
\address{Department of Mathematics, University of Hawaii, Honolulu, HI 96822 USA}
\email{yuen888@hawaii.edu}

\date{\today}

\begin{abstract}
We present an algorithm to compute all Borcherds product para\-modular
cusp forms of a specified weight and level, describing its
implementation in some detail and giving examples of its use.
\end{abstract}


\maketitle


\section{Introduction\label{sectionIntr}}

Borcherds products are a rich source of paramodular forms.
To begin with an example, consider the weight~$2$, level~$277$ nonlift
new paramodular cusp eigenform $\ftss \in \SwtKlevel2{277}$.
Here {\em nonlift\/} means {\em not a Gritsenko lift\/}.
This paramodular form, predicted by the paramodular conjecture of
A.~Brumer and K.~Kramer \cite{brkr14,brkr18}, shows the modularity of the
unique isogeny class of abelian surfaces~$A/\Q$ with conductor~$277$.
It was constructed in \cite{py15} as a quotient of polynomials in
Gritsenko lifts, with the proof that the quotient is holomorphic
requiring extensive computation in $\SwtKlevel8{277}$.
Our Borcherds product algorithm, the subject of this article,
produces an elegant alternative construction of~$\ftss$ as the sum of
a Borcherds product and a Gritsenko lift, which are constructed in
turn from Jacobi forms $\phi_1,\dotsc,\phi_9$ described below.
With $V_{\ell}:\Jackm\to \Jac{k}{m\ell}$ the level raising
operators of Eichler--Zagier \cite{ez85}, the first three Jacobi forms
define a weakly holomorphic Jacobi form~$\psi$ of weight~$0$ and index~$277$,
$$
\psi=-\dfrac{\phi_1|V_2}{\phi_1}-\dfrac{\phi_2|V_2}{\phi_2}
+\dfrac{\phi_3|V_2}{\phi_3}, \qquad
\psi(\tau,z)=\sum_{n,r\in\Z}c(n,r)q^n\zeta^r.
$$
Here $q=e^{2\pi i\tau}$ and $\zeta=e^{2\pi iz}$.
The Borcherds product constructed from~$\psi$ combines with a
Gritsenko lift constructed from the other six Jacobi forms
to give the nonlift new eigenform,
\begin{align*}
\ftss(\smallpvar)&=15\,q\,\zeta^{28}\xi^{277}
\prod_{ (m,n,r) \ge0}(1-q^n\zeta^r\xi^{mN})^{c(nm,r)} \\
&\quad+\Grit(-12\phi_4+2\phi_5+\phi_6+2\phi_7-2\phi_8-4\phi_9)(\smallpvar).
\end{align*}
Here $\xi=e^{2\pi i\omega}$.
The product is taken over $m,n,r\in\Z$ such that $m\ge0$, and if
$m=0$ then $n\ge0$, and if $m=n=0$ then $r<0$.
The nine Jacobi forms~$\phi_i$ in the construction of~$\ftss$ are
given as {\em theta blocks\/}, very useful functions due to
V.~Gritsenko, N.-P.~Skoruppa, and D.~Zagier \cite{gsz}.
Specifically, let $\eta(\tau)$ and $\vartheta(\tau,z)$ denote the
Dedekind eta function and the odd Jacobi theta function, and set
$\vartheta_\ell(\tau,z)=\vartheta(\tau,\ell z)$ for~$\ell\ge1$.
Then, with $0^e$ and $\ell^e$ abbreviating $\eta^e$
and~$(\vartheta_\ell/\eta)^e$,
\begin{alignat*}2
\phi_1 &=0^4 1^2 2^2 3^2 4^1 5^1 14^1 17^1
&\qquad\phi_2 &=0^4 1^1 3^1 4^2 5^1 6^1 8^1 9^2 15^1\\
\phi_3 &=0^4 1^1 2^1 3^1 4^2 5^1 7^1 8^1 9^1 17^1
&\qquad\phi_4 &=0^4 1^2 2^1 3^2 4^1 5^1 6^1 7^{-1} 9^1 14^1 15^1\\
\phi_5 &=0^4 1^2 2^3 3^1 4^1 11^1 13^1 15^1
&\qquad\phi_6 &=0^4 1^1 2^1 3^2 4^1 5^1 6^1 7^1 9^1 18^1\\
\phi_7 &=0^4 1^1 2^1 3^1 4^1 6^1 7^2 10^1 11^1 13^1
&\qquad\phi_8 &=0^4 1^1 3^2 4^1 7^1 8^2 10^1 11^2\\
\phi_9 &=0^4 1^2 2^1 3^1 4^1 6^{-1} 7^1 8^1 9^1 10^1 11^1 12^1.
\end{alignat*}
For instance, $\phi_4(\tau,z)=
\vartheta(\tau,z)^2\vartheta(\tau,3z)^2
\prod_{\ell\in\{2,4,5,6,9,14,15\}}\vartheta(\tau,\ell z)/
(\eta(\tau)^6\,\vartheta(\tau,7z))$.
See \cite{yps17b} for a full description of this example, 

In recent work \cite{psysqfree}, to study the spaces $\StwoKN$ of
weight~$2$ paramodular cusp forms of squarefree composite levels
$N<300$, we needed Borcherds products to help span the weight~$4$ spaces
for those levels, and we needed Borcherds products to construct the
weight~$2$ nonlift newforms of levels $N=249,295$ that were
predicted by the paramodular conjecture.
Using Borcherds products, we currently \cite{pyinflation} are carrying
out similar constructions for the prime levels $N<600$ where the
paramodular conjecture predicts nonlift newforms and prior work
has provided evidence that they exist \cite{py15}.
Work that the first and third authors of this article are preparing
with R.~Schmidt \cite{psysuper} uses Borcherds products to 
construct paramodular forms whose automorphic representations
have supercuspidal components.
As another example of the utility of having all Borcherds product cusp forms,
knowing that every Borcherds product in $\StwoKlevel{461}$ is also a Gritsenko
lift tells us that constructing the nonlift that the paramodular
conjecture predicts for this space requires other means.
With other applications for Borcherds products in mind as well, we
have created a tool to systematically construct all Borcherds product
paramodular cusp forms of a given weight and level.
This article describes the Borcherds product construction and its
implementation, and also the mathematical issues that arose in their context.
The method can produce Borcherds product paramodular noncusp forms as
well, but it needn't do so exhaustively.
The theory of Borcherds products for paramodular forms is given by
V.~Gritsenko and V.~Nikulin in~\cite{gn98}.

An entwining of theory, algorithm design, and experiment is required
to produce our mathematically rigorous method to find all Borcherds
products of a fixed weight and level in a practical amount of time.
Let a weight~$k$ and a level~$N$ be given,
both positive integers.  Let $\tau\in\UHP$ be a variable in the
complex upper half plane and $z\in\C$ a complex variable.  One part of our
algorithm produces all the Laurent polynomials
$$
\tilde\psi(\tau,z)=\sum_{n=\nmin}^{N/4+\nextra}\psi_n(\zeta)q^n
\in \LaurentQ[q,q\inv],
\quad q=e^{2\pi i\tau},\ \zeta=e^{2\pi iz},
$$
that determiningly truncate actual weight~$0$, index~$N$ weakly
holomorphic Jacobi forms,
with integral Fourier coefficients on singular indices,
whose resulting Borcherds products lie in the space of
paramodular cusp forms having the given weight and level,
$$
\psi(\tau,z)=\sum_{n=\nmin}^\infty\psi_n(\zeta)q^n\in\JzeroNwh,
\qquad
\BL(\psi)\in\SkKN.
$$
The polynomial-lengthening $\nextra$ in the penultimate display lets
the algorithm better avoid generating non-truncation Laurent
polynomials, which must be detected and discarded.  Also the longer
truncations let us check the cuspidality of the Borcherds products that
the algorithm produces, and compute more of their Fourier
coefficients.
A run of the algorithm finds those cuspidal Borcherds products of
weight~$k$ and level~$N$ that are further specified by two parameters $c$
and~$t$, which fix exponents in the variables~$\xi$ and $q$, 
$$
\BL(\psi)(\Omega)=q^{c+t}b(\zeta)(1-G(\zeta)q+\cdots)
\xi^{\BPc N}\exp(-\Grit(\psi)(\Omega)).
$$
Here the variable $\Omega=\smallpvar$ lies in the Siegel upper half
space~$\UHPtwo$, and $\xi=e^{2\pi i\omega}$.  The leading theta block
of the Borcherds product is $q^{c+t}b(\zeta)(1-G(\zeta)q+\cdots)$,
a Jacobi cusp form of weight~$k$ and index~$\BPc N$ denoted
$\phi\in\JkBPcNcusp$, and $\Grit$ denotes the Gritsenko lift,
$\Grit(\psi)(\Omega)=\sum_{m\ge1}(\psi|V_m)(\tau,z)\xi^{mN}$ with
each~$V_m$ an index-raising operator, so that
$\exp(-\Grit(\psi)(\Omega))=1-\psi(\tau,z)\xi^N+\cdots$.
The first nonzero Fourier--Jacobi coefficient of $\BL(\psi)$ is
$\phi$, and the next Fourier Jacobi coefficient is $-\phi\psi$,
an element of~$\Jcusp k{(\BPc+1)N}$.
This tells us to seek the source-form~$\psi$ of the Borcherds
product as a quotient $g/\phi$ with $g\in\Jcusp k{(\BPc+1)N}$
and $\phi\in\JkBPcNcusp$ a theta block having $q$-order $c+t$.
We know from~\cite{gsz} how to find all such theta blocks~$\phi$.
Our algorithm creates Laurent polynomial truncations~$\tilde g$ of
putative Jacobi cusp forms~$g$.  When these truncations are
long enough, the algorithm will generate only those Laurent
polynomial truncations~$\tilde\psi$ of~$\tilde g/\tilde\phi$ that determiningly
truncate actual forms~$\psi\in\JzeroNwh$; Theorem~\ref{divboundcor}
shows how to compute a sufficient length.  However, guaranteed-long-enough
truncations $\tilde g$ can be prohibitive computationally, so we use
shorter truncations, with the possibility of generating some extra 
polynomials $\tilde\psi$ that don't truncate actual forms~$\psi$.
Elementary theory of weakly holomorphic Jacobi forms lets us check
whether a candidate~$\tilde\psi$ really does truncate some $\psi$
in~$\JzeroNwh$, as shown in Proposition~\ref{testLaurentTrunc}.
In practice we tune the algorithm, aiming for long enough truncations
to avoid generating false~$\tilde\psi$ but using shorter truncations
than the guaranteeing length from the theory.
Theorem~\ref{cuspinesstheorem} gives a guaranteeing truncation length to
determine whether the paramodular form Borcherds products found by the
algorithm are cusp forms.  This length can be considerably greater than
necessary for the rest of the algorithm, leading to a second run.

Sections~\ref{sectionPara}, \ref{sectionJaco}, and~\ref{sectionThet}
give background on paramodular forms, Jacobi forms, and theta blocks.
Section~\ref{sectionBorc} quotes a version of the Gritsenko--Nikulin
theorem that gives conditions for a Borcherds product to be a
paramodular form, and it shows that only finitely many holomorphic
Borcherds products $\BL(\psi)$ can have a given leading theta block~$\phi$.
Section~\ref{sectionDivi} gives a sufficient truncation length to
prevent our algorithm from generating false candidates~$\tilde\psi$.
Section~\ref{sectionConf} gives an algorithm to test whether a
candidate $\tilde\psi$ truncates an actual element of~$\JzeroNwh$.
Section~\ref{sectionCusp} gives a sufficient truncation length to
determine whether the Borcherds products found by the algorithm are cusp forms.
Section~\ref{sectionBLAlgorithm} presents the algorithm to find all
Borcherds products.
Section~\ref{sectionExam} gives examples of using the algorithm.

\section{Paramodular forms\label{sectionPara}}

\subsection{Definitions\label{sectionParaDefs}}

We introduce notation and terminology for paramodular forms.
The degree~$2$ symplectic group $\Spgp(2)$ of $4\times4$ matrices is
defined by the condition $g'Jg=J$, where the prime denotes matrix
transpose and $J$ is the skew form $\smallmat0{-1}1{\phantom{-}0}$
with each block $2\times2$.  The map
$\iota:\SLgp(2)\times\SLgp(2)\lra\Spgp(2)$ given by
$$
\cmat{a_1}{b_1}{c_1}{d_1}\times\cmat{a_2}{b_2}{c_2}{d_2}
\longmapsto
\left[\begin{array}{cc|cc}
a_1 & 0 & b_1 & 0 \\ 0 & a_2 & 0 & b_2 \\
\hline
c_1 & 0 & d_1 & 0 \\ 0 & c_2 & 0 & d_2
\end{array}\right]
$$
is a homomorphism.
The Klingen parabolic subgroup of~$\Spgp(2)$ is
$$
\Ptwoone=\lset
\left[\begin{array}{cc|cc}
* & 0 & * & * \\ * & * & * & * \\
\hline
* & 0 & * & * \\ 0 & 0 & 0 & *
\end{array}\right]\rset,
$$
with either line of three zeros forcing the remaining two
in consequence of the matrices being symplectic.
The map $\iota_1:\SLgp(2)\lra\Ptwoone$ is the restriction
of~$\iota$ that takes each $g\in\SLgp(2)$ to~$\iota(g,1_2)$.
For any positive integer~$N$, the paramodular group $\KN$ of degree~$2$
and level~$N$ is the group of rational symplectic matrices that
stabilize the column vector lattice $\Z\oplus\Z\oplus\Z\oplus N\Z$.  In
coordinates,
$$
\KN=\lset
\left[\begin{array}{cc|cc}
* & *N & * & * \\ * & * & * & */N \\
\hline
* & *N & * & * \\ *N & *N & *N & *
\end{array}\right]\in\SptwoQ:\text{all $*$ entries integral}
\rset.
$$
Here the upper right entries of the four subblocks are ``more integral
by a factor of~$N$'' than implied immediately by the definition of the
paramodular group as a lattice stabilizer, but
the extra conditions hold because the matrices are symplectic.

Let $\UHPtwo$ denote the Siegel upper half space of $2\times2$
symmetric complex matrices that have positive definite imaginary part,
generalizing the complex upper half plane~$\UHP$.
Elements of this space are written
$$
\Omega=\pvar\in\UHPtwo,
$$
with $\tau,\omega\in\UHP$, $z\in\C$, and $\Im\Omega>0$.
Also, letting $\e(w)=e^{2\pi iw}$ for~$w\in\C$, our standard notation is
$$
q=\e(\tau),\quad\zeta=\e(z),\quad\xi=\e(\omega).
$$
The real symplectic group $\SptwoR$ acts on~$\UHPtwo$ as fractional
linear transformations, $g(\Omega)=(a\Omega+b)(c\Omega+d)\inv$ for
$g=\smallmatabcd$, and the Siegel factor of automorphy is
$j(g,\Omega)=\det(c\Omega+d)$.  Fix an integer~$k$.
Any function $f:\UHPtwo\lra\C$ and any real symplectic matrix
$g\in\SptwoR$ combine to form another such function through the
weight~$k$ operator, $f\wtk g(\Omega)=j(g,\Omega)^{-k}f(g(\Omega))$.
A paramodular form of weight~$k$ and level~$N$ is a holomorphic
function $f:\UHPtwo\lra\C$ that is $\wtk{\KN}$-invariant; the K\"ocher
Principle says that for any positive $2\times2$ real matrix~$Y_o$, the
function $f\wtk g$ is bounded on $\lset\Im \Omega>Y_o\rset$ for all
$g\in\SptwoQ$.  The space of weight~$k$, level~$N$ paramodular
forms is denoted~$\MkKN$.

The Witt map~$W$ takes each pair $(\tau,\omega)$ in $\UHP\times\UHP$
to the matrix $\smallmat\tau00\omega$ in~$\UHPtwo$,
and its pullback~$W^*$ takes each function~$f$ on~$\UHPtwo$ to the function
$(W^*\!f)(\tau,\omega)=f(\smallmat\tau00\omega)$ on~$\UHP\times\UHP$.
Especially, $W^*$ takes $\MkKN$ to
$\MFs\SLtwoZ\otimes\MFs\SLtwoZ\wtk{\smallmat N001}$,
with $\MFs\SLtwoZ$ the space of weight~$k$, level~$1$ elliptic modular
forms.  Siegel's $\Phi$ map takes any holomorphic function that has a
Fourier series of the form
$f(\Omega)=\sum_t\fc tf\,\e(\ip t\Omega)$, summing over rational
positive semidefinite $2\times2$ matrices~$t$, to the function $(\Phi
f)(\tau)=\lim_{\omega\to i\infty}(W^*\!f)(\tau,\omega)$.
A paramodular form $f$ in~$\MkKN$ is a cusp form if $\Phi(f\wtk g)=0$
for all~$g\in\SptwoQ$; the space of such cusp forms is denoted~$\SkKN$.

\subsection{Symmetric and antisymmetric forms\label{sectionSymm}}

A paramodular form of level $N$ has a Fourier expansion
$$
f(\Omega)=\sum_{t\in\XtwoNsemi}\fc tf\,\e(\ip t\Omega)
$$
where
$\XtwoNsemi=\lset\smallpindN:n,m\in\Znn,r\in\Z,4nmN-r^2\ge0\rset$
and $\ip t\Omega=\tr(t\Omega)$.  A paramodular form is a cusp
form if and only if its Fourier expansion is supported on $\XtwoN$,
defined by the strict inequality $4nmN-r^2>0$; this characterization
of cusp forms does not hold in general for groups commensurable
with~$\SptwoZ$, but it does hold for~$\KN$.
Consider any $\SptwoR$ matrix of the form
$g=\alpha\matoplus\alpha^*=\smallmat\alpha00{\alpha^*}$
with $\alpha\in\GLtwoR$, where the superscript asterisk denotes matrix
inverse-transpose.  Introduce the notation $t[u]=u'tu$ for compatibly
sized matrices $t$ and~$u$.  Then we have $f\wtk g(\Omega)=(\det\alpha)^k
\sum_{t\in\XtwoNsemi[\alpha]}\fc{t[\alpha\inv]}f\,\e(\ip t\Omega)$
for any paramodular form~$f$; especially, if $g$ normalizes~$\KN$,
so that $f\wtk g$ is again a paramodular form, then
$\fc t{f\wtk g}=(\det\alpha)^k\fc{t[\alpha\inv]}f$ for~$t\in\XtwoNsemi$.
Let $\GampmsupzeroN$ denote the subgroup of~$\GLtwoZ$ defined by the
condition $b=0\mymod N$, where elements are denoted $\smallmatabcd$.
For $\alpha\in\GampmsupzeroN$, the matrix $g=\alpha\inv\matoplus\alpha'$
lies in~$\KN$ and we get
$\fc{t[\alpha]}f=(\det\alpha)^k\fc tf$ for~$t\in\XtwoNsemi$.

The elliptic Fricke involution,
$\alpha_N=\tfrac1{\sqrt N}\smallmat0{-1}N{\phantom{-}0}:
\tau\mapsto-\tfrac1{N\tau}$,
normalizes the level~$N$ Hecke subgroup $\GzN$ of~$\SLtwoZ$, and it
squares to~$-1$ as a matrix, hence to the identity as a transformation.
The corresponding paramodular Fricke involution is
$\mu_N=\alpha_N^*\matoplus\alpha_N:
\smallpvar\longmapsto\smallmat{\omega N}{-z}{-z}{\tau/N}$.
Thus $\fc t{f\wtk{\mu_N}}=\fc{t[\alpha_N']}f$ for any $f$ in~$\MkKN$.
The paramodular Fricke involution normalizes the paramodular group
$\KN$ and squares to~$1$ as a transformation.
The space $\CFs\KN$ decomposes as the direct sum of the Fricke eigenspaces
for the two eigenvalues~$\pm1$, $\SkKN=\SkKN^+\oplus\SkKN^-$.
We let $\epsilon$ denote either eigenvalue.

The Fourier coefficients of a paramodular Fricke eigenform
$f\in\MkKN^\epsilon$ satisfy the condition $\fc\smallpindnN f
=\epsilon\,\fc{\smallmat{\phantom{-}n}{-r/2}{-r/2}{\phantom{-}mN}}f$,
and because $\beta\matoplus\beta\in\KN$ where
$\beta=\left[\begin{smallmatrix}-1&0\\\phantom{-}0&1\end{smallmatrix}\right]$,
so that $f\wtk{\beta\matoplus\beta}=f$, they also satisfy the condition
$\fc{\smallmat{\phantom{-}n}{-r/2}{-r/2}{\phantom{-}mN}}f
=(-1)^k\fc\smallpindN f$.
The two conditions combine to give the {\em involution conditions\/}
on the Fourier coefficients,
$$
\fc\smallpindnN f=(-1)^k\epsilon\,\fc\smallpindN f.
$$
Decompose~$f$ as a $(q,\xi)$-Fourier series whose $(n,m)$-co\-ef\-fi\-cients
are functions of~$z$,
$$
f(\Omega)=\sum_{n,m}F_{n,m}(\zeta)q^n\xi^{mN},
\qquad
F_{n,m}(\zeta)=\sum_r\fc\smallpindN f\zeta^r.
$$
The involution conditions become $F_{m,n}=(-1)^k\epsilon\,F_{n,m}$ on
the $(n,m)$-coefficients.
A paramodular Fricke eigenform is {\em symmetric\/} if $(-1)^k\epsilon=+1$,
and {\em antisymmetric\/} if $(-1)^k\epsilon=-1$.
For an antisymmetric form, the involution conditions imply that
$\fc{\smallmat n{r/2}{r/2}{nN}}f=0$ for all $n$ and~$r$,
and $F_{n,n}(\zeta)=0$ for all~$n$.

\subsection{Fourier--Jacobi expansion\label{sectionFour}}

The Fourier--Jacobi expansion of a para\-mod\-u\-lar cusp form
$f\in\SkKN$ is
$$
f(\Omega)=\sum_{m\ge1}\phi_m(f)(\tau,z)\xi^{mN},
\quad \Omega=\pvar,\ \xi=\e(\omega),
$$
with Fourier--Jacobi coefficients
$$
\phi_m(f)(\tau,z)=\sum_{t=\smallpindN\in\XtwoN}\fc tfq^n\zeta^r,
\quad q=\e(\tau),\ \zeta=\e(z).
$$
Here the coefficient $\fc tf$ is also written $\fcJ{n,r}{\phi_m}$, and
the sum taken over pairs $(n,r)$ such that $4nmN-r^2>0$.
Each Fourier--Jacobi coefficient $\phi_m(f)$ lies in the space
$\JkmNcusp$ of weight~$k$, index~$mN$ Jacobi cusp forms, whose
dimension is known (Jacobi forms will briefly be reviewed just below).
These are Jacobi forms of level one---this is an advantage of the
paramodular group over the Hecke subgroup $\Gamma_0^{(2)}(N)$
of~$\SptwoZ$---and trivial character, both omitted from the notation.
The additive (Gritsenko) lift $\Grit:\JkNcusp\lra\SkKN^\epsilon\subset\SkKN$
for $\epsilon=(-1)^k$ is a section of the map $\SkKN\lra\JkNcusp$
that takes each~$f$ to~$\phi_1(f)$, \ie, $\phi_1({\rm Grit}(\phi))=\phi$
for all $\phi\in\JkNcusp$.

\section{Jacobi forms\label{sectionJaco}}

For the theory of Jacobi forms, see \cite{ez85,gn98,sz89}.
We give basics for quick reference.

\subsection{Definitions, singular bounds and principal bounds}

Let $k$ be an integer and let $\JFind$ be a nonnegative integer.  The
complex vector spaces of weight~$k$, index~$\JFind$ Jacobi forms,
Jacobi cusp forms, and weakly holomorphic Jacobi forms consist of
holomorphic functions $g:\UHP\times\C\lra\C$ that have Fourier series
representations
$$
g(\tau,z)=\sum_{n,r}\fcJ{n,r}gq^n\zeta^r,\quad\text{all $\fcJ{n,r}g\in\C$},
$$
and that satisfy transformation laws and constraints on the support.
With the usual notation $\gamma(\tau)=(a\tau+b)/(c\tau+d)$
and $j(\gamma,\tau)=c\tau+d$ for $\gamma=\smallmatabcd\in\SLtwoZ$ and
$\tau\in\UHP$, the transformation laws are
\begin{itemize}
\item $g(\gamma(\tau),z/j(\gamma,\tau))
=j(\gamma,z)^k\e(\JFind cz^2/j(\gamma,\tau))g(\tau,z)$ for all
$\gamma\in\SLtwoZ$,
\item $g(\tau,z+\lambda \tau+\mu)
=\e(-\JFind \lambda^2\tau-2\JFind \lambda z)g(\tau,z)$ for all
$\lambda,\mu\in\Z$.
\end{itemize}
Equivalently, the function
$$
\JFtoSMFm g:\UHPtwo\lra\C,\qquad
(\JFtoSMFm g)(\smallpvar)=g(\tau,z)\e(\JFind\omega)
$$
is holomorphic, has Fourier series representation
$(\JFtoSMFm g)(\Omega)=\sum_{n,r}\fcJ{n,r}gq^n\zeta^r\xi^\JFind$
where $\xi=\e(\omega)$,
and transforms as a Siegel modular form of weight~$k$ under the subgroup
of~$\SptwoZ$ generated by $\iota_1(\SLtwoZ)$ and~$-1_4$ and 
by the Heisenberg subgroup, given by the products of matrices
$\smallmat a00{a^*}$,
where $a=\smallmat10{\lambda}1$ with $\lambda\in\Z$,
times matrices $\smallmat1b01$,
where $b=\smallmat0{\mu}{\mu}{\kappa}$ with $\mu,\kappa\in\Z$;
this group is $\PtwooneZ$.
The quadratic character
$v_H\left(\smallmat a00{a^*}\smallmat1b01\right)=(-1)^{\lambda+\mu+\kappa}$
on the Heisenberg subgroup extends to $\PtwooneZ$
by making it trivial on $\iota_1(\SLtwoZ)$ and~$-1_4$.
To describe the constraints on the support,
associate to any integer pair $(n,r)$ the discriminant
$$
D=D(n,r)=4n\JFind-r^2.
$$
The principal part of~$g$ is $\sum_{n<0}g_n(\zeta)q^n$ where
$g_n(\zeta)=\sum_r\fcJ{n,r}g\zeta^r$, and the singular part is
$\sum_{D(n,r)\le0}\fcJ{n,r}gq^n\zeta^r$.
The transformation law
$(\JFtoSMFm g)\wtk{\smallmat a00{a^*}}=\JFtoSMFm g$
where $a=\smallmat10{\lambda}1$ for any ${\lambda}\in\Z$
shows that
$\fcJ{n-{\lambda}r+{\lambda}^2\JFind,r-2{\lambda}\JFind}g =\fcJ{n,r}g$
for all $(n,r)$ and~${\lambda}$,
and also $D(n-{\lambda}r+{\lambda}^2\JFind,r-2{\lambda}\JFind)=D(n,r)$;
thus for positive index~$\JFind$, all Fourier coefficients
$\fcJ{n,r}g$ having a given discriminant~$D$ are determined by those
coefficients having discriminant~$D$ such that furthermore
$|r|\le\JFind$, or even $-\JFind\le r<\JFind$.
\begin{itemize}
\item For the space $\Jac k\JFind$ of Jacobi forms, if $\JFind>0$ then
  the sum is taken over integers $n$ and~$r$ such that $D\ge0$, so
  that in particular $n\ge0$, and if $\JFind=0$ then the sum is taken
  over pairs $(n,r)\in\Znn\times\{0\}$, and we have elliptic modular forms.
\item For the space $\Jcusp k\JFind$ of Jacobi cusp forms, if $\JFind>0$ then
  the sum is taken over integers $n$ and~$r$ such that $D>0$, so
  that in particular $n>0$, and if $\JFind=0$ then the sum is taken
  over pairs $(n,r)\in\Zpos\times\{0\}$, and we have elliptic cusp forms.
\item For the space $\Jwh k\JFind$ of weakly holomorphic Jacobi forms
  the sum is taken over integers $n\gg-\infty$ and~$r$.
  A weakly holomorphic Jacobi form is holomorphic on $\UHP\times\C$.
  Assuming now that the index~$\JFind$ is positive, we show that the
  conditions $n\gg-\infty$ and $D\gg-\infty$ are equivalent.
  For one direction, if for some~$n_o$, all coefficients $\fcJ{n,r}g$
  where $n<n_o$ are~$0$, then all coefficients $\fcJ{n,r}g$ where
  $4n\JFind-r^2<4n_o\JFind-\JFind^2$ are~$0$: indeed, we may take
  $|r|\le\JFind$, giving
  $4n\JFind-\JFind^2\le4n\JFind-r^2<4n_o\JFind-\JFind^2$ and thus
  $n<n_o$, so $\fcJ{n,r}g=0$ as claimed.  Conversely, if for
  some~$D_o$, all coefficients $\fcJ{n,r}g$ where $4n\JFind-r^2<D_o$
  are~$0$, then also $\fcJ{n,r}g=0$ for all $n<D_o/(4\JFind)$.
  Weakly holomorphic Jacobi forms $g(\tau,z)$ of index~$0$ are
  constant in~$z$ by Liouville's Theorem, and so only their Fourier
  coefficients $\fcJ{n,0}g$ can be nonzero.  Thus they are like
  elliptic modular forms other than possibly being meromorphic
  at~$i\infty$---for example, Klein's modular invariant function~$j$.
\end{itemize}
We make two more comments about weakly holomorphic Jacobi forms.
First, the singular part $\sum_{D(n,r)\le0}\fcJ{n,r}gq^n\zeta^r$ of
such a Jacobi form is determined by the finitely many nonzero singular Fourier
coefficients $\fcJ{n,r}g$ such that $n\le\JFind/4$ and $|r|\le\JFind$;
indeed, we know that it is determined by its terms that have indices
$(n,r)$ with $|r|\le\JFind$, and this combines with the condition
$D\le0$ to give $n\le\JFind/4$.
Second, consider the Fourier series of such a Jacobi form as a
$q$-expansion,
$$
g(\tau,z)=\sum_ng_n(\zeta)q^n,\qquad
g_n(\zeta)=\sum_r\fcJ{n,r}g\zeta^r\text{ for each~$n$}.
$$
The coefficient $g_n(\zeta)$ is a Laurent polynomial in~$\zeta$ because 
the support of~$g$ is bounded by $n\ge n_o$ for some~$n_o$, and so the 
$q^n$-coefficient $\fcJ{n,r}g\zeta^r$ can be nonzero
only for the finitely many values of~$r$ such that
$4n\JFind-r^2\ge4n_o\JFind-\JFind^2$.

When we need the more general case of a Jacobi form $g$,
possibly of half-integral weight and/or index, that
transforms by a multiplier, we indicate so in the notation;
for example, the odd Jacobi theta function~$\vartheta$ lies
in~$\Jcusp{1/2}{1/2}(\epsilon^3 v_H)$, where $\epsilon$ is the
multiplier of the Dedekind eta function; \cf~\cite{gn98}.
When all the Fourier coefficients lie in a ring we also append the
ring to the notation; for example, $\Jwh 0\JFind(\Z)$ denotes the
$\Z$-module of weight~$0$, index~$\JFind$ weakly holomorphic forms
of trivial multiplier with integral Fourier coefficients.

\subsection{Determining Fourier coefficients for weight zero}

The following elementary result provides a starting point for our algorithm.

\begin{theorem}\label{wt0integralitythm}
Let $m$ be a positive integer.
Any weight~$0$, index~$\JFind$ weakly holomorphic Jacobi form
$\psi\in\Jwh0\JFind$,
$$
\psi(\tau,z)=\sum_{n,r}\fcJ{n,r}\psi q^n\zeta^r=\sum_n\psi_n(\zeta)q^n,
$$
is determined by its Fourier coefficients $\fcJ{n,r}\psi$ for all
pairs $(n,r)$ such that $4n\JFind-r^2<0$.  Consequently, $\psi$ is
determined by its coefficient functions $\psi_n(\zeta)$ for~$n<\JFind/4$.
\end{theorem}

\begin{proof}
It suffices to prove that $\Jzerom=0$ for $m>0$.
This fact is stated on page~11 of \cite{ez85}, but we give an
independent proof.
Let $g\in\Jzerom$ be given, and define a function $f:\UHP\times\Q\lra\C$
by $f(\tau,\lambda)=\e(m\lambda^2\tau)g(\tau,\lambda\tau)$.
For any fixed~$\lambda$, the function $\tau\mapsto f(\tau,\lambda)$ is
an elliptic modular form of weight~$0$ for a subgroup of~$\SLtwoQ$
commensurable with $\SLtwoZ$, so it is a constant~$C(\lambda)$.
The Fourier expansion $g(\tau,z)=\sum_{n,r}\fcJ{n,r}g q^n\zeta^r$,
summing over integer pairs $(n,r)$ with $4nm-r^2\ge0$, gives
$f(\tau,\lambda)=\sum_{n,r}\fcJ{n,r}gq^{n+\lambda r+\lambda^2m}$.
We show that for any fixed integer pair $(n_o,r_o)$ with
$4n_o\JFind-r_o^2\ge0$, the Fourier coefficient $\fcJ{n_o,r_o}g$ is~$0$.
For small enough nonzero~$\lambda$ we have $n_o+\lambda r_o+\lambda^2
m\ne0$, and hence the coefficient of~$q^{n_o+\lambda r_o+\lambda^2m}$
in the Fourier expansion of~$g$ is~$0$.  This coefficient is
$\sum_{n,r:n+\lambda r=n_o+\lambda r_o}\fcJ{n,r}g$.
We show that after shrinking~$\lambda$ further if necessary, this
sum is the singleton $\fcJ{n_o,r_o}g$, which therefore is~$0$ as
desired.  Indeed, consider a second integer pair $(n,r)$ with
$4n\JFind-r^2\ge0$ and $n+\lambda r=n_o+\lambda r_o$; thus $n\ne n_o$
since $\lambda$ is nonzero.  If $n\le n_o-1$ then the estimate
$n_o+\lambda r_o=n+\lambda r\le n+2\lambda\sqrt{n\JFind}
\le n_o-1+2\lambda\sqrt{(n_o-1)\JFind}$
gives a contradiction for small enough~$\lambda=\lambda(n_o,r_o)$,
independently of~$(n,r)$.
If $n\ge n_o+1$ then the estimate
$n_o+\lambda r_o=n+\lambda r\ge n-2\lambda\sqrt{n\JFind}
=\sqrt n(\sqrt n-2\lambda\sqrt{\JFind})
\ge\sqrt{n_o+1}(\sqrt{n_o+1}-2\lambda\sqrt{\JFind})
=n_o+1-2\lambda\sqrt{(n_o+1)\JFind}$ again gives a contradiction for small
enough $\lambda=\lambda(n_o,r_o)$.
\end{proof}

We usually use a weaker version of Theorem~\ref{wt0integralitythm}:
Any $\psi\in\Jwh0\JFind$ is determined by its singular part, \ie, by
its Fourier coefficients $\fcJ{n,r}\psi$ for all pairs $(n,r)$ such
that $4n\JFind-r^2\le0$, and so $\psi$ is determined by its
coefficient functions $\psi_n(\zeta)$ for~$n\le\JFind/4$.  We proved
the stated theorem partly in case the argument here that $\Jzerom=0$
for~$m>0$, a stronger fact than $\Jcusp0\JFind=0$ for~$m>0$, is not
well known.

\section{Theta blocks\label{sectionThet}}

The theory of theta blocks is due to Gritsenko, Skoruppa, and Zagier
\cite{gsz}.

\subsection{Eta and theta\label{sectionThetDef}}

Recall the Dedekind eta function $\eta\in\Jcusp{1/2}0(\epsilon)$
and the odd Jacobi theta function
$\vartheta\in\Jcusp{1/2}{1/2}(\epsilon^3 v_H)$,
\begin{align*}
\eta(\tau)&=q^{1/24}\prod_{n\ge1}(1-q^n),\\
\vartheta(\tau,z)
&=\sum_{n\in\Z}(-1)^nq^{(n+1/2)^2/2}\zeta^{n+1/2}\\
&=q^{1/8}(\zeta^{1/2}-\zeta^{-1/2})
\prod_{n\ge1}(1-q^n\zeta)(1-q^n\zeta\inv)(1-q^n).
\end{align*}
For any $r\in\Zpos$, define 
$\vartheta_r\in\Jcusp{1/2}{r^2/2}(\epsilon^3 v_H^r)$
to be $\vartheta_r(\tau,z)=\vartheta(\tau,rz)$,
so that
$$
\vartheta_r(\tau,z)/\eta(\tau)
=q^{1/12}(\zeta^{r/2}-\zeta^{-r/2})
\prod_{n\ge1}(1-q^n\zeta^r)(1-q^n\zeta^{-r}).
$$
The quotient $\vartheta_r/\eta$ lies
in~$\Jweak0{r^2/2}(\epsilon^2 v_H^r)$, where
as in \cite{ez85}, {\em weak\/} Jacobi forms are supported on $n\ge0$.
As shown by their product formulas, $\eta(\tau)$ is nonzero for
all~$\tau\in\UHP$ and $\vartheta_r(\tau,z)/\eta(\tau)$ vanishes
precisely when $z+\Lambda_\tau$ is an $r$-torsion point of the abelian
group $E_\tau=\C/(\tau\Z+\Z)$.

\begin{definition}\label{tbdef}
A {\bf theta block} is a meromorphic function of the form
$$
\TB(\tau,z)=\eta(\tau)^{\varphi(0)}
\prod_{r\ge1}(\vartheta_r(\tau,z)/\eta(\tau))^{\varphi(r)},
$$
where $\varphi:\Z\lra\Z$ is even and finitely supported.
\end{definition}
Here we designate two subtypes of theta block, to be used throughout
this article:
A theta block such that $\varphi(r)\ge0$ for each~$r\in\Zpos$ is a
theta block {\bf without denominator}.
A theta block is {\bf basic} if it is a weakly holomorphic Jacobi form of
integral weight~$k$ and nonnegative integral index~$\JFind$.

Thus a theta block has the product form
\begin{equation}\label{TBwithdoubleproduct}
\TB(\tau,z)=q^\BPA b(\zeta)\prod_{n\ge1,\,r\in\Z}(1-q^n\zeta^r)^{\varphi(r)},
\end{equation}
where the leading exponent of~$q$ is
$$
\BPA=\frac1{24}\sum_{r\in\Z}\varphi(r),
$$
and the {\bf baby theta block} of~$\TB$ is
$b(\zeta)=\prod_{r\ge1}(\zeta^{r/2}-\zeta^{-r/2})^{\varphi(r)}$, or
$$
b(\zeta)=\zeta^{-B}\prod_{r\ge1}(\zeta^r-1)^{\varphi(r)},
\qquad
B=\frac12\sum_{r\ge1}r\varphi(r).
$$
The multiplicity function~$\varphi$ determines a {\bf germ}
$$
G(\zeta)=\sum_{r\in\Z}\varphi(r)\zeta^r.
$$
This germ determines the $q$-expansion  $1-G(\zeta)q+\cdots$ 
of the double product in~\eqref{TBwithdoubleproduct}, 
which has coefficients in  $\LaurentZ$,  
and overall the theta block is
$$
\TB(\tau,z)=q^Ab(\zeta)(1-G(\zeta)q+\cdots).
$$
For the grand theta block formula that expresses
this $q$-expansion in terms of double partitions
and the functions~$G(\zeta^j)$, see \cite{psyan}.

Functions that we call {\em atoms} were introduced in \cite{gsz} 
to characterize holomorphic theta blocks.
Let $\mu$ denote the M\"obius function from elementary number theory.
For any positive integer~$r$, define the $r$-th atom to be
$$
\atom r(\tau,z)
=\prod_{s\mid r}\vartheta_s(\tau,z)^{\mu(r/s)}
=\vartheta_r(\tau,z)
\frac1{\prod_{p\mid r}\vartheta_{r/p}(\tau,z)}
\prod_{p \ne p':\,p,p'\mid r}\vartheta_{r/pp'}(\tau,z)
\cdots.
$$
Because $\atom r$ for~$r\ge2$ is the theta block with $\varphi(s)=\mu(r/s)$
for~$s\mid r$ and $\varphi(s)=0$ for all other nonnegative~$s$, the
theta block product form \eqref{TBwithdoubleproduct} specializes to give
$$
\atom r(\tau,z)=\zeta^{-\tfrac12\phi(r)}\Phi_r(\zeta)
\prod_{s\mid r}\prod_{n:\gcd(n,s)=1}
\Phi_{r/s}(q^n\zeta^s)\Phi_{r/s}(q^n\zeta^{-s}),
\quad r\ge2,
$$
where $\phi$ is Euler's totient function and $\Phi_d$ is the $d$-th
cyclotomic polynomial.
(A small point here is that for~$d>1$, the monic $d$-th cyclotomic
polynomial also has constant term~$1$, and for~$d=1$ the sign-change
$\Phi_1(X)=-(1-X)$ is absorbed into the previous display's multiplicand.)
For $r \ge 2$, we have $\atom r \in \Jweak{0}{m}(v_H^{\phi(r)})$ for 
$m=\tfrac12 r^2\prod_{p\mid r}(1-1/p^2)$.
By M\"obius inversion, the penultimate display gives
$$
\vartheta_r=\prod_{s\mid r}\atom s,
$$
and so the formal representation of a theta block by atoms is
$$
\TB(\tau,z)=\eta(\tau)^{\nu(0)}
(\vartheta(\tau,z)/\eta(\tau))^{\nu(1)}
\prod_{r\ge2}\atom r(\tau,z)^{\nu(r)},
$$
where
$$
\nu(0)=\varphi(0),\qquad
\nu(r)=\sum_{t\ge1}\varphi(tr)\text{ for }r\ge1.
$$
Also,
the baby theta block of~$\TB$ is
$$
b(\zeta)=\zeta^{-B}
\prod_{r\ge1}\Phi_r(\zeta)^{\nu(r)},
\qquad
B=\frac12\sum_{r\ge1}\phi(r)\nu(r),
$$
with $B$ the same here as in the previous paragraph.
The condition for a theta block to be holomorphic in~$(\tau,z)$ is
that $\nu(r)\ge0$ for~$r\ge1$, \ie, $\sum_{t\ge1}\varphi(tr)\ge0$
for~$r\ge1$.  This is also the condition for its baby theta block to
be holomorphic in~$z$.
In terms of cyclotomic polynomials and the atom-multiplicity
function~$\nu$, the theta block product form is
$$
\TB(\tau,z)=q^\BPA b(\zeta)
\prod_{n\ge1}\Phi_1(q^n)^{\nu(0)}
\prod_{r\ge1}\prod_{s\mid r}\prod_{(n,s)=1}
\Phi_{r/s}(q^n\zeta^s)^{\nu(r)}\Phi_{r/s}(q^n\zeta^{-s})^{\nu(r)},
$$
with $\BPA=\tfrac1{24}\nu(0)+\tfrac1{12}\nu(1)$.

Recall that a theta block is basic if it is a weakly holomorphic
Jacobi form of integral weight~$k$ and nonnegative integral index~$\JFind$.
A basic theta block has weight $k=\tfrac12\varphi(0)=\tfrac12\nu(0)$
and index $m=\tfrac12\sum_{r\ge1}r^2\varphi(r)
=\tfrac12\sum_{r\ge1}\nu(r)r^2\prod_{p\mid r}(1-1/p^2)$,
the last equality holding because
$\varphi(r)=\sum_{t\ge1}\mu(t)\nu(tr)$ for $r\ge1$.
Because a basic theta block is holomorphic on $\UHP\times\C$, it
satisfies the holomorphy condition given at the end of the previous paragraph.
All theta blocks in this article are basic, some without denominator and
others with.

\subsection{Weak holomorphy from basic theta blocks}

Let a weight~$k$ and an index~$\JFind$ be given.
Let $\phi\in\Jwh k\JFind$ be a basic theta block without denominator,
and let $\BPA$ denote the $q$-order of~$\phi$.
Let $V_2$ denote the index-raising Hecke operator of \cite{ez85}, page~$41$.
Not only does the quotient $(\phi|V_2)/\phi$ transform as a Jacobi form of
weight~$0$ and index~$\JFind$, but furthermore it is weakly
holomorphic and has integral Fourier coefficients.
Theorem~1.1 of~\cite{gpy15} says, in paraphrase, that if $\BPA$ is even,
or if $\BPA$ is odd and $\phi$ lies in $\Jac k\JFind$, then the
Borcherds product arising from the signed quotient
$\psi=(-1)^\BPA(\phi|V_2)/\phi$ is a {\em holomorphic\/} paramodular form.
Theorem~6.6 of~\cite{gpy15} provides more specifics.
The following theorem determines when a basic theta block~$\phi$, now
possibly {\em with\/} denominator, gives rise to a weakly holomorphic Jacobi
form by the same quotient construction.  The theorem connotes no
holomorphy assertion about the resulting Borcherds product.

\begin{theorem}\label{whquotienttheorem}
For a given weight~$k \in \Z$ and index~$\JFind \in \Znn$, consider a
basic theta block $\phi\in\Jwh k\JFind$, having $q$-order $\BPA \in \Z$,
baby theta block~$b$, and germ~$G$,
$$
\phi(\tau,z)=q^\BPA b(\zeta)(1-G(\zeta)q+\cdots).
$$
Let the decomposition of the basic theta block as a product of atoms be
$$
\phi(\tau,z)=\eta(\tau)^{\nu(0)-\nu(1)}
\prod_{r\ge1}\atom r(\tau,z)^{\nu(r)}.
$$
The following conditions are equivalent.
\begin{enumerate}
\item $(\phi|V_2)/\phi$ is weakly holomorphic, 
\item $b(\zeta)$ divides $b(\zeta^2)$ in $\LaurentC$,
\item $\nu(r)\ge\nu(2r)$ for all $r\ge1$.
\end{enumerate}
In the affirmative case we have $(\phi|V_2)/\phi \in\Jwh 0\JFind(\Z)$.
\end{theorem}

The proof of this theorem requires some elementary theory of divisors.

For any positive integer~$r$ and any integer row vector $v=[A\quad B]\in\Z^2$,
introduce a map $T(r,v)$ that enhances a complex upper half plane
point~$\tau$ with a complex pre-image of an $r$-torsion point in
$E_\tau=\C/\Lambda_\tau$ where $\Lambda_\tau=\tau\Z+\Z$,
$$
T(r,v):\UHP\lra\UHP\times\C,\qquad
T(r,v)(\tau)=(\tau, (A\tau+B)/r).
$$
Thus the image $V(r,v)$ of~$T(r,v)$ is the zero set in $\UHP\times\C$
of the polynomial
$$
\varpi(r,v;\tau,z)=A\tau+B-rz,
$$
an irreducible holomorphic convex algebraic subvariety of
dimension one.
By the Taylor expansion, any holomorphic function vanishing
on~$V(r,v)$ is divisible by $\varpi(r,v;\tau,z)$ in the ring of
holomorphic functions.
We call the points $(\tau,z)$ of~$V(r,v)$ torsion points or
$r$-torsion points of $\UHP\times\C$, or primitive $r$-torsion points
if $z+\Lambda_\tau$ has order~$r$ in~$E_\tau$,
and we call $V(r,v)$ a torsion curve.
The points of $V(r,v)$ have order $r/\gcd(r,A,B)$, and $V(r,v)$ is
unaffected by dividing $r$ and~$v$ by the gcd, giving new $r$ and~$v$
such that $V(r,v)$ consists of primitive $r$-torsion points.

As noted above, the zero set of $\vartheta_r$ is the set of
$r$-torsion points of~$\UHP\times\C$.  Any theta block
$\phi=\eta^{\varphi(0)}\prod_{r\ge1}(\vartheta_r/\eta)^{\varphi(r)}$
is thus a meromorphic function whose zeros and poles are $R$-torsion
points where $R$ is the least common multiple of the nonzero support
of~$\varphi$.
Only a finite number of curves $V(r,v)$ with $r\mid R$ meet any
compact set in~$\UHP\times\C$; furthermore, distinct $V(r,v)$ are
disjoint, and so any point $(\tau_o,z_o)\in V(r_o,v_o)$ has a
neighborhood~$U_o$ that meets no other torsion curve $V(r,v)$ with
$r\mid R$.  The product form of theta blocks shows that they all
belong to the following set~$\mathcal F$.

\begin{definition}\label{merof}
Let $\mathcal F$ be the multiplicative group of meromorphic functions~$f$
on~$\UHP\times\C$ whose zero and polar sets consist only of torsion,
each such torsion point $(\tau_o,z_o)$ lying in a unique torsion curve
$V(r,v)$ and having a neighborhood~$U_o$ where $f$ takes the form
$$
f(\tau,z) = \varpi(r,v;\tau,z)^{\nu}h_o(\tau,z),
\quad\nu\in\Z,\ h_o\text{ holomorphic and nonzero on }U_o.
$$
\end{definition}

The integer $\nu$ in Definition~{\ref{merof}} depends only on the
curve $V(r,v)$ and not on the point $(\tau_o,z_o)\in V(r,v)$, because
$\nu$ is locally constant as a function of $(\tau_o,z_o)$ and $V(r,v)$
is convex.  Thus, a meromorphic function $f\in\mathcal F$ has a well
defined order $\ord(f,V(r,v))=\nu$ on each torsion curve $V(r,v)$.
The divisor of~$f$ is
$$
\divv(f)=\sum_{r,v:\, \gcd(r,v)=1}\ord(f,V(r,v))V(r,v),
$$
an element of the free $\Z$-module on the distinct torsion curves.
In particular, $\divv(\vartheta_r)=\sum_vV(r,v)$ and
$\divv(\atom r)=\sum_{v:\,\gcd(r,v)=1}V(r,v)$;
these are the $r$-torsion and primitive $r$-torsion divisors.
The $r$-torsion divisor is the sum over positive integers~$s$
dividing~$r$ of the primitive $s$-torsion divisors, consonantly with
the relation $\vartheta_r=\prod_{s\mid r}\atom s$.
The divisor of a theta block
$\phi=\eta^{\nu(0)-\nu(1)}\prod_{r\ge1}(\atom r)^{\nu(r)}$ is
\begin{equation}\label{thetablockdivisor}
\divv(\phi)=\sum_{r\ge1}\sum_{v:\,\gcd(r,v)=1}\nu(r)V(r,v).
\end{equation}
For the general divisor theory of holomorphic functions, see
\cite{gunning}, pp.~76--78.
Here we describe only the simpler divisor theory of~$\mathcal F$,
giving direct computational arguments even when more general ones are
available.

For any $2\times2$ integral matrix having positive determinant,
$\sigma=\smallmatabcd\in\MatsplustwoZ$ with $\Delta=\det\sigma>0$,
define a corresponding integral symplectic matrix with
similitude~$\Delta$,
$$
\iotaonesigma=\left[\begin{matrix}
a & & b & \\
& \Delta & & \\
c & & d & \\
  & & & 1 \\
\end{matrix}\right].
$$
That is, $(\iotaonesigma)'J\iotaonesigma=\Delta J$
where $J$ is the skew form $\smallmat0{-1}1{\phantom{-}0}$.
The matrix $\sigma$ acts holomorphically on~$\UHP$ while
$\iotaonesigma$ acts biholomorphically on~$\UHP\times\C$,
$$
\sigma(\tau)=\frac{a\tau+b}{c\tau+d},\qquad
(\iotaonesigma)(\tau,z)=\left(\sigma(\tau),\frac{\Delta z}{c\tau+d}\right).
$$

We observe the action of $\iota_1\MatsplustwoZ$ on the torsion curves $V(r,v)$.
Let $\Adj(\sigma)$ denote the adjoint
$\smallmat{\phantom{-}d}{-b}{-c}{\phantom{-}a}$ of a matrix
$\sigma=\smallmatabcd$ in~$\MatsplustwoZ$.
The diagram
$$
\xymatrix{
\UHP\ar[rr]^-{T(r,v)}\ar[d]_\sigma && \UHP\times\C\ar[d]^{\iotaonesigma}\\
\UHP\ar[rr]^-{T(r,v\Adj\sigma)} && \UHP\times\C
}
$$
commutes, both ways around taking any~$\tau\in\UHP$ to
$(\sigma(\tau),(\Delta/r)(A\tau+B)/(c\tau+d))$.  Consequently,
$$
\text{$\iotaonesigma$ takes $V(r,v)$ to~$V(r,v\Adj(\sigma))$}.
$$
Replacing $\sigma$ by~$\Adj(\sigma)$ in the previous display shows
that $\iotaoneAdjsigma$ takes $V(r,v)$ to~$V(r,v\sigma)$, and so also,
because $\iotaonesigma\inv$ acts as $\iotaoneAdjsigma$ followed by
$(\tau,z)\mapsto(\tau,z/\Delta)$,
$$
\text{$\iotaonesigma\inv$ takes $V(r,v)$ to~$V(\Delta r,v\sigma)$}.
$$

We show that $\iota_1\MatsplustwoZ$ preserves the group~$\mathcal F$
from Definition~\ref{merof}
and preserves vanishing order under corresponding images of torsion.
Consider any $f\in\mathcal F$, $r\in\Zpos$, $v\in\Z^2$, and
$\sigma\in\MatsplustwoZ$.
To show that $f\circ\iotaonesigma$ again lies in~$\mathcal F$, first
note that we have established that $\iotaonesigma\inv$ takes
$r$-torsion to $\Delta r$-torsion, and so $f\circ\iotaonesigma$ has
its zero and polar sets supported on torsion because $f$ does.
Second, for any torsion point $p_o\in V(r,v)$, set
$p_1=\iotaonesigma(p_o)\in V(r,v\Adj(\sigma))$.
We have $f=\varpi(r,v\Adj(\sigma))^\nu h_1$ on some neighborhood~$U_1$
of~$p_1$, where $h_1$ is a holomorphic unit and
$\nu=\ord(f,V(r,v\Adj(\sigma)))=\ord(f,\iotaonesigma(V(r,v)))$.
Accordingly, $f\circ\iotaonesigma
=\left(\varpi(r,v\Adj(\sigma))\circ\iotaonesigma\right)^\nu
h_1\circ\iotaonesigma$ on the neighborhood $U_o=\iotaonesigma\inv U_1$
of~$p_o$.  A small calculation gives $\varpi(r,v\Adj(\sigma))\circ\iotaonesigma
=\Delta/(c\tau+d)\cdot\varpi(r,v)$, and so, because
$\varpi(r,v\Adj(\sigma))\circ\iotaonesigma$ and $\varpi(r,v)$ differ
by a holomorphic unit on~$U_o$, we have $f\circ\iotaonesigma\in\mathcal F$.
This argument has also shown that
$$
\ord(f\circ\iotaonesigma,{V(r,v)})=\ord(f,\iotaonesigma(V(r,v)).
$$

\begin{lemma}
The curves of primitive $r$-torsion form one orbit under~$\PtwooneZ$.
\end{lemma}

\begin{proof}
The Heisenberg transformations $(\tau,z)\mapsto(\tau,z+\lambda\tau+\mu)$,
for $\lambda,\mu\in\Z$, send $V(r,v)$ to $V(r,v_1)$
with $v_1=v+r[\lambda\,\,\mu]\equiv v\bmod r$.
The negative identity fixes all torsion curves.
The $\iotaone \SLtwoZ$ subgroup acts by
$(\iotaonesigma)(V(r,v))=V(r,v\Adj(\sigma))$.
Thus the curves of primitive $r$-torsion are stable under $\PtwooneZ$
since it is generated by these three types.
To complete the proof we  map a general primitive $r$-torsion curve
$V(r,v)$ to $V(r,[0\,\,1])$.  

Use $\gcd(r,v)=1$ to select $[c \,\, d] \equiv v \bmod r$ with
$\gcd(c,d)=1$; one way to achieve this is to use Dirichlet's theorem
on primes in arithmetic progression.  Heisenberg transformations
show that $V(r,[c\,\,d])$ is in the orbit of $V(r,v)$.
Select $a,b \in \Z$ so that $\sigma=\smallmatabcd\in\SLtwoZ$.  Now
$(\iotaonesigma)(V(r,[c\,\,d]))=V(r,[c\,\,d]\Adj(\sigma))=V(r,[0\,\,1])$
is also in the orbit.  
\end{proof}

\begin{corollary}\label{cor3}
Let $m,k\in\frac12\Znn$ and $r \in \Zpos$.
If the divisor of a weakly holomorphic Jacobi form $g\in\Jwh {k}{m}(\chi)$ 
contains one curve of primitive $r$-torsion, 
then the divisor of $g$ contains all curves of 
primitive $r$-torsion, and $g/\atom r$ is holomorphic on $\UHP \times \C$.
\end{corollary}

\begin{proof}
If the divisor of~$g$ contains one curve~$V(r,v_o)$ of primitive $r$-torsion, 
then the automorphy of $g$ with respect to $\PtwooneZ$ implies that the 
divisor of $g$ contains the entire $\PtwooneZ$-orbit
$$\sum_{v\in \Z^2 :\,\gcd(r,v)=1}V(r,v) = \divv(\atom r).$$
Thus $g/\atom r$ is holomorphic on $\UHP \times \C$. 
\end{proof}

The next lemma characterizes divisibility by a holomorphic theta block 
in the ring of weakly holomorphic Jacobi forms in terms of divisibility by its baby theta block 
in the ring of holomorphic functions.
\begin{lemma}\label{newlemmafour}
Let $k_1, k_2 \in \frac12\Z$, and $m_1, m_2 \in \frac12\Znn$.
Let $g\in\Jwh {k_1}{m_1}(\chi_1)$ be a weakly holomorphic Jacobi form
and let $\phi\in\Jwh {k_2}{m_2}(\chi_2)$ 
be a theta block with baby theta block~$b$.
The following equivalence holds:
$$
\frac{g}{\phi}\in\Jwh{k_1-k_2}{m_1-m_2}(\chi_1\chi_2\inv)
\quad\text{if and only if}\quad
\text{$\frac{g}{b}$ is holomorphic on $\UHP\times\C$}.
$$
\end{lemma}

\begin{proof}
Assume that ${g}/{\phi}\in\Jwh {k_1-k_2}{m_1-m_2}(\chi_1\chi_2\inv)$.
Using the infinite product for the theta block~$\phi$, we have
$\phi(\tau,z)=b(\zeta)h_o(\tau,z)$ with $h_o$ holomorphic on $\UHP \times \C$.
Thus, ${g}/{b}=({g}/{\phi}) h_o$  is holomorphic, as asserted.

Assume that ${g}/{b}$ is holomorphic on $\UHP \times \C$.  
Let the decomposition of the holomorphic theta block~$\phi$ as a
product of atoms be  
$$
\phi(\tau,z)=\eta(\tau)^{\nu(0)-\nu(1)}
\prod_{r\ge1}\atom r(\tau,z)^{\nu(r)}, 
\quad \text{$\nu(r) \ge 0$ for $r \ge 1$.}
$$
By induction on the number of atoms, it is enough to prove the case where $\phi$ 
is a single atom. 
The base case when $m_2=0$ is simple because $b=1$ and $\phi=\eta^{\nu(0)}\in\Jwh {k}{0}(\chi)$ is a holomorphic unit 
with $k=\frac12 \nu(0)$ and $\chi=\epsilon^{\nu(0)}$.   
Therefore we may assume that the
theta block is $\phi=\atom r \in\Jwh {0}{m_2}(\chi_2)$ and the 
baby theta block is  $b(\zeta) = \zeta^{-\phi(r)/2}\Phi_r(\zeta)$.  
There is a point $(\tau_o,z_o) \in \UHP \times\C$ where $b(\zeta_o)=0$ for $\zeta_o =e(z_o)$; 
for example, we may take $z_o=1/r$.  
The curve $V(r,v_o)$ for $v_o=[0\, 1]$ passes through $(\tau_o,z_o) $ and 
$b(\zeta) = \varpi(r,v_o;\tau,z)b_o(\tau,z)$ for a holomorphic~$b_o$ with $b_o(\tau_o,z_o) \ne 0$ 
because $\zeta_o$ is a simple root of the cyclotomic polynomial.  
Since $g/b$ is holomorphic, the divisor of~$g$ includes $V(r,v_o)$.  
By Corollary~\ref{cor3}, 
$g/\atom r$ is  holomorphic on $\UHP \times \C$ 
and necessarily of weight~{$k_1$}, index~{$m_1-m_2$}, and multiplier $\chi_1\chi_2\inv$.  
Finally, $g/\atom r$ has a finite principal part because both $g$ and~$\atom r$ do.  
\end{proof}

Now we prove Theorem~\ref{whquotienttheorem}.  
Recall its statement: three conditions are equivalent,
$(1)$~$(\phi|V_2)/\phi$ is weakly holomorphic,
$(2)$~$b(\zeta)$ divides $b(\zeta^2)$ in $\LaurentC$,
and $(3)$~$\nu(r)\ge\nu(2r)$ for all $r\ge1$;
and $(\phi|V_2)/\phi\in\Jwh0\JFind(\Z)$ when these conditions hold.

\begin{proof}
By Lemma~\ref{newlemmafour},  $(\phi|V_2)/\phi$ is weakly holomorphic if and only if 
 $(\phi|V_2)/b$ is holomorphic.   The action of~$V_2$ is:  
\begin{equation}\label{veetwo}
(\phi|V_2)(\tau,z)=2^{k-1}\phi(2\tau,2z)
+\tfrac12\big(\phi(\tfrac12\tau,z)+\phi(\tfrac12\tau+\tfrac12,z)\big).
\end{equation}
The last two terms have the same baby theta block as~$\phi$,
and so their quotient by~$b$ is holomorphic on~$\UHP\times\C$.  
The baby theta block of the first term is~$b(\zeta^2)$, so if 
 $b(\zeta)$ divides $b(\zeta^2)$ in $\LaurentC$ then  $(\phi|V_2)/b$ is holomorphic.  
 
 On the other hand if  $(\phi|V_2)/b$ is holomorphic then $\phi(2\tau,2z)/b(\zeta)$ is holomorphic.  
 For any root of unity $\zeta_o=e(z_o)$ such that $b(\zeta_o)=0$, 
 the theta block product form \eqref{TBwithdoubleproduct} shows that
 this quotient is
$b(\zeta^2)/b(\zeta)$ times a holomorphic unit on some neighborhood
of any~$(\tau_o,z_o)$, and so any discontinuity of $b(\zeta^2)/b(\zeta)$
at~$\zeta_o$ is removable.  Thus $b(\zeta^2)/b(\zeta)$
is a rational function in $\zeta\in\C\setminus\{0\}$ without poles,
and hence it lies in~$\LaurentC$.  
Thus conditions $(1)$ and~$(2)$ are equivalent. 

We show that conditions $(2)$ and~$(3)$ are equivalent. 
We have
$$
b(\zeta)=\zeta^{-B}
\prod_{r\ge1}\Phi_r(\zeta)^{\nu(r)},
\text{ where }B=\frac12\sum_{r\ge1}\phi(r)\nu(r),
$$
and the elementary observation that $\Phi_r(X^2)=\Phi_{2r}(X)$ if $r$
is even and $\Phi_r(X^2)=\Phi_{2r}(X)\Phi_r(X)$ if $r$ is odd gives
\begin{align*}
b(\zeta^2)
&=\zeta^{-2B}\prod_{r\ge1}\Phi_{2r}(\zeta)^{\nu(r)}
\prod_{r\ge1\text{ odd}}\Phi_r(\zeta)^{\nu(r)}
=b(\zeta)\zeta^{-B}\prod_{r\ge1}\Phi_{2r}(\zeta)^{\nu(r)-\nu(2r)}.
\end{align*}
The equation $b(\zeta^2)/b(\zeta)
=\zeta^{-B}\prod_{r\ge1}\Phi_{2r}(\zeta)^{\nu(r)-\nu(2r)}$ 
then gives the claimed equivalence.  
Because the $\Phi_r$
have disjoint divisors and nonzero roots,
this product is holomorphic on $\C\setminus\{0\}$ if and only if 
$\nu(r)\ge\nu(2r)$ for all~$r\ge1$.  

If  $(\phi|V_2)/\phi$ is holomorphic then its Fourier expansion is given by 
formal division of Fourier expansions.  By equation~\ref{veetwo} 
and the first displayed expansion of~$\phi$ in
the statement of Theorem~\ref{whquotienttheorem}, if 
$b(\zeta)$ divides $b(\zeta^2)$  in $\LaurentC$, hence in $\LaurentZ$, then this quotient has integral coefficients and so 
$(\phi|V_2)/\phi \in\Jwh 0\JFind(\Z)$. 
\end{proof}

\subsection{Finding all basic theta blocks\label{sectionFind}}

Consider a fixed weight~$k$ and level~$N$, where we seek all Borcherds
products $f=\BL(\psi)$ in~$\SkKN$.
The Borcherds product algorithm receives parameters $c$ and~$t$, where
the leading $\xi$-power in~$f$ is~$\xi^{cN}$, and the leading $q$-power
of the leading theta block is $q^A$ with $A=c+t$.  Set $m=cN$.  The
algorithm needs to traverse all basic theta blocks in $\Jcusp km$ that
have leading $q$-power~$q^A$; these are the possible leading theta
blocks of the sought Borcherds products.

We first show how to search systematically for such basic theta blocks
without denominator, as this is faster and suffices for some purposes,
even if not for finding all Borcherds products.
Consider basic theta blocks in ``$\varphi$-form'' from
section~\ref{sectionThetDef}, with $\varphi$ understood to be
nonnegative on~$\Zpos$.  Thus we need $\varphi(0)=2k$ and
$\sum_{r\ge1}\varphi(r)=12A-k$, the sum being the number of
$\vartheta$'s in the basic theta block.  The index~$m$ condition is
$\sum_{r\ge1}r^2\varphi(r)=2m$.
The computer search seeks basic theta blocks
$\eta^{\varphi(0)}\prod_{i=1}^\ell(\vartheta_{x_i}/\eta)$ where
$\ell=12A-k$ and $\sum_{i=1}^\ell x_i^2=2m$ with
$x_1\ge\cdots\ge x_\ell\ge1$.
Given such an $\ell$-tuple of $x_i$-values, $\varphi(r)$ for~$r\ge1$
is the number of~$x_i$ that equal~$r$.

To search systematically for all basic theta blocks in $\Jcusp km$
that have leading $q$-power~$q^A$, including any such basic theta
blocks with denominator, consider basic theta blocks in ``$\nu$-form,''
since any holomorphic theta block must be a product of atoms.  
Thus we need $\nu(0)=2k$ and
$\nu(1)=12A-k$, this being the {\em net\/} number of $\vartheta$'s in
the basic theta block, counting denominator-$\vartheta$'s negatively.  The
index~$m$ condition is
$\sum_{r\ge2}\nu(r)r^2\prod_{p\mid r}(1-1/p^2)=2m-\nu(1)$.
The computer search seeks basic theta blocks
$\eta^{\nu(0)}(\vartheta/\eta)^{\nu(1)}\prod_{i=1}^s\atom{x_i}$
such that $\sum_{i=1}^sx_i^2\prod_{p\mid x_i}(1-1/p^2)=2m-\nu(1)$ and
$x_1\ge\cdots\ge x_s\ge2$.
Given such an $s$-tuple of $x_i$-values, $\nu(r)$ for~$r\ge2$
is the number of~$x_i$ that equal~$r$.
Here we don't have as tidy a bound on~$s$ as we had for~$\ell$ in the
previous paragraph, but the worst case, when all~$x_i$ are~$2$, is
$s=(2m-\nu(1))/3$.

Thus it is a finite algorithm to find all basic theta blocks for any
given $k$, $N$, $c$ and~$t$.  For each basic theta block found, we compute
its valuation \cite{gsz} (to be discussed in section~\ref{sectionVal})
to determine whether it is a Jacobi cusp form.

Our algorithm to find Borcherds products of weight~$k$ and level~$N$, 
where $A=c+t$ and $C=cN$, will search for basic theta blocks of
weight~$k$ and index~$m=C$ that are cusp forms.
The relations $12A-k=\nu(1)\ge0$ and $A=c+t$ give $(k-12c)/12\le t$.
The relations $\nu(1)=12A-k$ and $2m-\nu(1)\ge0$ give
$2cN\ge12t+12c-k$, so that $t\le(k-(12-2N)c)/12$.
For $N\le5$, the bounds
$$
c\ge1,
\qquad
\max\lset\frac{k-12c}{12},0\rset\le t\le\frac{k-(12-2N)c}{12}
$$
show that basic theta blocks relevant to the desired
Borcherds products can exist only in a discrete quadrilateral
of pairs~$(c,t)$.  We will make use of these bounds for $(k,N)=(46,4)$
in section~\ref{kN=464}.

\section{Borcherds products\label{sectionBorc}}

\subsection{Borcherds product theorem\label{sectionBPTh}}

The Borcherds product theorem, quoted here from~\cite{psysqfree},
is a special case of Theorem~3.3 of \cite{gpy15}, which
in turn is quoted from \cite{gn98,grit12} and relies on the work
of R.~Borcherds.

\begin{theorem}\label{BPthm}
Let $N$ be a positive integer.  Let $\psi\in\JzeroNwh$ be a weakly
holomorphic weight~$0$, index~$N$ Jacobi form, having Fourier expansion
$$
\psi(\tau,z)=\sum_{\substack{n,r\in\Z\\n\gg-\infty}}c(n,r)q^n\zeta^r
\quad\text{where }q=\e(\tau),\ \zeta=\e(\zeta).
$$
Define
\begin{alignat*}2
A&=\frac1{24}c(0,r)+\frac1{12}\sum_{r\ge1}c(0,r),
 &\quad 
B\phantom{_0}&=\frac12\sum_{r\ge1}r\,c(0,r),\\
C&=\frac12\sum_{r\ge1}r^2c(0,r),
 &\quad
D_0&=\sum_{n\le-1}\sigma_0(|n|)c(n,0).
\end{alignat*}
Suppose that the following conditions hold:
\begin{enumerate}
\item $c(n,r)\in\Z$ for all integer pairs $(n,r)$ such that $4nN-r^2\le0$,
\item $A\in\Z$,
\item $\sum_{j\ge1}c(j^2nm,jr)\ge0$ for all primitive integer
  triples $(n,m,r)$ such that $4nmN-r^2<0$ and $m\ge0$.
\end{enumerate}
Then for weight $k=\frac12c(0,0)$ and Fricke eigenvalue
$\epsilon=(-1)^{k+D_0}$, the Borcherds product $\BL(\psi)$ lies
in~$\MFs\KN^\epsilon$.
For sufficiently large $\lambda$, for $\Omega=\smallmat\tau
zz\omega\in\UHPtwo$ and $\xi=\e(\omega)$, the Borcherds product has
the following convergent product expression on the subset
$\lset\Im \Omega>\lambda I_2\rset$ of~$\UHPtwo$:  
$$
\BL(\psi)(\Omega)=
q^A \zeta^B \xi^C 
\prod_{\substack{n,m,r\in \Z,\ m\ge0\\
\text{if $m=0$ then $n\ge0$}\\
\text{if $m=n=0$ then $r<0$}}}
(1-q^n\zeta^r\xi^{mN})^{c(nm,r)}.
$$
Also, let $\varphi(r)=c(0,r)$ for~$r\in\Znn$, and recall the
corresponding basic theta block,
$$
\TB(\varphi)(\tau,z)=\eta(\tau)^{\varphi(0)}
\prod_{r\ge1}(\vartheta_r(\tau,z)/\eta(\tau))^{\varphi(r)}
\qquad\text{where }\vartheta_r(\tau,z)=\vartheta(\tau,rz).
$$
On $\lset\Im \Omega>\lambda I_2\rset$ the Borcherds product is a
rearrangement of a convergent infinite series,
$$
\BL(\psi)(\Omega)
=\TB(\varphi)(\tau,z)\xi^C\exp\left(-\Grit(\psi)(\Omega)\right).  
$$
\end{theorem}

In the theorem, the divisor of the Borcherds product $\BL(\psi)$ is a sum
of Humbert surfaces with multiplicities, the multiplicities
necessarily nonnegative for holomorphy.  Let $\KN^+$ denote the
supergroup of~$\KN$ obtained by adjoining the paramodular Fricke
involution.  The sum in item~(3) of the theorem is the multiplicity of
the following Humbert surface in the divisor,
$$
{\rm Hum}(4nmN-r^2,r)=\KN^+\lset\Omega\in\UHPtwo:
\ip\Omega{\smallmat{n}{r/2}{r/2}{mN}}=0\rset.
$$
This surface lies in~$\KN^+\bs\UHPtwo$.  As the notation in the
display suggests, this surface depends only on the discriminant
$D=4nmN-r^2<0$ and on~$r$, so that we may take a matrix with $m=1$,
and furthermore it depends only on the residue class of~$r$
modulo~$2N$; this result is due to Gritsenko and Hulek \cite{gh98}.
We use it to parametrize Humbert surfaces as ${\rm Hum}(D,r)$, taking
for each such surface a suitable $\smallmat{n}{\tilde r/2}{\tilde r/2}N$
with $4nN-\tilde r^2=D$ and $\tilde r=r\mymod2N$.  (The Humbert
surface in the previous display has other denotations, such as
$H_N(-D,r)$ in \cite{gpy15}.)  We note that condition~(1) in
Theorem~\ref{BPthm} imposes conditions upon the singular Fourier
coefficients of~$\psi$, while the Humbert  multiplicity condition~(3)
refers only to coefficients such that $D<0$.

A holomorphic cuspidal Borcherds product has a basic cuspidal leading
theta block.  Indeed, every nontrivial Fourier-Jacobi coefficient of a
paramodular cusp form  is a Jacobi cusp form with positive integral
index and weight.

\subsection{Borcherds product anatomy\label{sectionBLAanatomy}}

Consider a holomorphic Borcherds product~$f$, \ie, $f=\BL(\psi)$ for some
$\psi\in\JzeroNwh$, connoting that the singular coefficients of~$\psi$
are integral.  The source-form $\psi$ determines an even finitely supported
multiplicity function
$$
\varphi:\Z\lra\Z,\qquad\varphi(r)=\fcJ0r
\ \text{where }\psi(\tau,z)=\sum_{n,r}\fcJ nrq^n\zeta^r,
$$
which in turn determines a Jacobi form basic theta block,
\begin{align*}
\TB(\tau,z)&=\eta(\tau)^{\varphi(0)}
\prod_{r\ge1}\big(\vartheta_r(\tau,z)/\eta(\tau)\big)^{\varphi(r)}
=q^Ab(\zeta)(1-G(\zeta)q+\cdots),
\end{align*}
whose germ is the $q^0$-coefficient of~$\psi$
(noting that $\fcJ{0,-r}\psi=\fcJ{0,r}\psi$ for any positive integer~$r$),
$$
G(\zeta)=\varphi(0)+\sum_{r\ge1}\varphi(r)(\zeta^r+\zeta^{-r}).
$$
The initial $q$-exponent of the leading basic theta block is~$A$,
and the initial $\xi$-exponent of $\BL(\psi)$ is an integer multiple of~$N$,
$$
C=\frac12\sum_{r\ge1}r^2\varphi(r).
$$
Again with $V_\ell$ the index-raising Hecke operator of \cite{ez85},
though now extended to weakly holomorphic Jacobi forms
\cite{gn98}, the Gritsenko lift of~$\psi$ is
$$
\Grit(\psi)(\Omega)=\psi(\tau,z)\xi^N+(\psi|V_2)(\tau,z)\xi^{2N}
+(\psi|V_3)(\tau,z)\xi^{3N}+\cdots.
$$
The Borcherds product theorem gives the expansion
$$
f(\Omega)=\TB(\tau,z)\xi^C\operatorname{exp}(-\Grit(\psi)(\Omega))
=\TB(\tau,z)\xi^C(1-\psi(\tau,z)\xi^N+\cdots).
$$
This shows that, taking $C=\BPc N$, the first two Fourier--Jacobi
coefficients of the Borcherds product $f=\BL(\psi)$ are the basic theta
block determined by the source form~$\psi$, and the additive inverse
of that theta block multiplied by~$\psi$,
\begin{align*}
\phi_\BPc(f)(\tau,z)&=\TB(\tau,z),\\
\phi_{\BPc+1}(f)(\tau,z)&=-\psi(\tau,z)\TB(\tau,z).
\end{align*}
Thus, as in \cite{gn98}, we can read off the source weakly holomorphic
Jacobi form as the negative quotient of the first two Fourier--Jacobi
coefficients,
$$
\psi=-\phi_{\BPc+1}/\phi_\BPc.
$$
Also the Borcherds product expansion of~$f$ shows that
$F_{n,m}(\zeta)$ can be nonzero only for~$m\ge\BPc=C/N$.
By the involution condition $F_{m,n}(\zeta)=(-1)^k\epsilon F_{n,m}(\zeta)$,
also $F_{n,m}$ can be nonzero only for~$n\ge\BPc$.
The Borcherds Product Theorem result 
$(-1)^k\epsilon=(-1)^{D_0}$, shows that the symmetry or antisymmetry
of the Borcherds product is determined by the parity of~$D_0$.
Sometimes we can get $D_0$ quite early in a computation from 
a short initial expansion of~$\psi$.  

\subsection{Finitely many holomorphic Borcherds products have a given
  leading theta block\label{sectionFiniteLifts}}

We show that only finitely many holomorphic Borcherds products
$f=\BL(\psi)$ can have a given leading theta block~$\phi$, necessarily basic.  
This fact guarantees that a particular step
in our algorithm to find all Borcherds products is finite.  
Combined with section~\ref{sectionFind}, this also shows that each 
$\MwtKlevel kN$ contains only finitely many Borcherds products.  

Consider the partial order $\prec$ on~$\R^n$ such that $v\prec w$ if
$w-v\in\Rnn^n$ and $w\ne v$.
One readily shows that any sequence in~$\Znn^n$ contains a subsequence
that is increasing (possibly not strictly) at each coordinate, and
consequently any sequence in~$\Znn^n$ that takes infinitely many
values contains a subsequence that is strictly increasing in the
partial order.  In particular, any infinite set in~$\Znn^n$
contains two elements that are in order, $v_1\prec v_2$.

\begin{theorem}\label{finitelymanyBPthm}
Given a weight~$k$, a level~$N$, and a basic theta block $\phi\in \Jac{k}{cN}$,
at most finitely many weakly holomorphic Jacobi forms $\psi\in\JzeroNwh$
give rise to holomorphic Borcherds products $\BL(\psi)\in\MwtKlevel kN$ with
leading theta block~$\phi$.
\end{theorem}

The theorem doesn't give an effective upper bound on the number of
Borcherds products having a given leading theta block.  In practice,
we search for the finitely many weakly holomorphic Jacobi forms
$\psi\in\JzeroNwh$ having $q^0$-coefficient $G(\zeta)$ where $G$ is
the germ of the basic theta block, and having nonnegative Humbert
multiplicities.

\begin{proof}
The given basic theta block takes the form
$\phi(\tau,z)=q^\BPA b(\zeta)(1-G(\zeta)q+\cdots)$,
with $A$ a nonnegative integer since $\phi$ is a Jacobi form.  
Consider any weakly holomorphic
Jacobi form $\psi$ in $\JzeroNwh$ whose resulting Borcherds product is holomorphic
and has leading theta block~$\phi$, especially the weight of $\BL(\psi)$ is that of~$\phi$.  
The involution conditions give us $\BPc\le\BPA$, and so there is nothing to prove unless this
relation holds.
The $q$-order of~$\phi$ is~$A$ and the $q$-order of the second Fourier-Jacobi coefficient of $\BL(\psi)$ 
is, by the involution conditions $F_{m,n}=(-1)^k\epsilon\,F_{n,m}$, at least~$c$.  
Therefore~$\psi$, which is the additive inverse of the quotient, 
has $q$-order bounded below by~$c-A$.  
As discussed in section~\ref{sectionJaco}, consequently the
discriminant $D=4nN-r^2$ is bounded below for nonzero Fourier
coefficients $\fcJ{n,r}\psi$, and the Fourier coefficients for a given
discriminant~$D$ are determined by the Fourier coefficients for~$D$
such that furthermore $|r|\le N$.  Thus overall there are a finite
number of possibilities $(D,r)$ with $|r|\le N$ for the nonzero
singular Fourier coefficients of the~$\psi$ under consideration, and hence only a finite
collection of possible Humbert surfaces supported in the divisor of these $\BL(\psi)$.  Place these
Humbert surfaces in some order, so that their multiplicities form a
vector $v(\psi)\in\Znn^h$.  Consider any $\psi_1,\psi_2$ such that
$v(\psi_1)\preceq v(\psi_2)$, referring to the partial order discussed
just above.  This relation makes the quotient paramodular form
$\BL(\psi_2)/\BL(\psi_1)$ of weight~$0$ and level~$N$ holomorphic
because its multiplicity at each Humbert surface is nonnegative.  But
$\MwtKlevel0N=\C$, and so the two $\BL(\psi_i)$ are proportional,
making the two $\psi_i$ equal because they arise as the quotients of
the first two nonzero Fourier--Jacobi coefficients.  Overall, the
condition $\psi_1\prec\psi_2$ is impossible if $\BL(\psi_1)$ and
$\BL(\psi_2)$ are holomorphic and have leading theta block~$\phi$.
As observed just before this theorem, it follows that only finitely
many~$\psi$ can give rise to such a Borcherds product. 
\end{proof}

\section{Divisibility by a basic theta block\label{sectionDivi}}

\subsection{Existence of a determining truncation bound}
Let $k$ be an integer and let $m$ be a positive integer.
In the context of our algorithm, $m$ will be $(c+1)N$, where we seek
Borcherds products in $\MkKN$ having $\xi$-order~$cN$, but here $m$ is
general for simplicity.
The Jacobi cusp form space $\Jkmcusp$ has a basis of power series
in~$q$ whose coefficients are Laurent polynomials in~$\zeta$
over~$\Z$, elements of the form
$$
g(\tau,z)=\sum_{n=1}^\infty g_n(\zeta)q^n,\quad
\text{each }g_n(\zeta)\in\LaurentZ.
$$
We store determining truncations of the basis,
$$
\tilde g(\tau,z)=\sum_{n=1}^\nmax g_n(\zeta)q^n,\quad
\text{each }g_n(\zeta)\in\LaurentZ.
$$
Let $\phi\in\Jcusp k{m'}$ be a basic theta block, where $0<m'<m$,
having baby theta block~$b$,
$$
\phi(\tau,z)=q^Ab(\zeta)(1-G(\zeta)q+\cdots).
$$
In our algorithm, $m'$ will be~$cN$.
Consider a $\Jkmcusp$ element $g(\tau,z)$ as displayed above, but no
longer assumed to be an element of the basis whose truncations we store.
%
%
By Lemma~\ref{newlemmafour}, 
$g(\tau,z)$ is divisible by~$\phi(\tau,z)$ in the ring of weakly holomorphic Jacobi forms 
exactly when $g(\tau,z)$ is divisible by~$b(\zeta)$ in the ring of holomorphic functions. 
By the Fourier expansion in~$q$, this last condition is that each of the 
coefficients $g_n(\zeta)$ is divisible in $\LaurentC$ by~$b(\zeta)$. 
This section shows that for a computable $\nmax=\nmax(b)$,
checking the divisibility of all the coefficients $g_n(\zeta)$
by~$b(\zeta)$ reduces to checking the divisibility for
$n=1,\dotsc,\nmax$.  That is, the divisibility by~$\phi$ of any linear combination of basis elements
$g$ can be checked using only their truncations $\tilde g$
out to this~$\nmax$.

The existence of $\nmax$ is immediate.  The subspace of $\Jkmcusp$
elements that are divisible by~$b(\zeta)$ is the nested intersection
$$
\lset g\in\Jkmcusp:b(\zeta)\mid g(\zeta)\rset
=\bigcap_{\ell \ge1}\lset g\in\JkNcusp:
b(\zeta)\mid g_n(\zeta)\text{ for }n=1,\dotsc,\ell \rset,
$$
with the divisibility of $g_n(\zeta)$ by~$b(\zeta)$ being checked
in $\LaurentC$.
The sequence of spaces being intersected stabilizes because their
dimensions form a decreasing sequence in~$\Znn$, and so there exists a
least integer $\nmax$ as sought,
$$
\lset g\in\JkNcusp:b(\zeta)\mid g(\zeta)\rset
=\lset g\in\JkNcusp:b(\zeta)\mid g_n(\zeta)
\text{ for }n=1,\dotsc,\nmax\rset.
$$
We want to compute an upper bound of this $\nmax$  
uniformly in terms of $b$.  

\subsection{Laurent polynomial division\label{sectionLaur}}

Before addressing the problem, we state a  division algorithm for~$\LaurentZ$.

\begin{proposition}[Laurent polynomial division algorithm\label{lpdprop}]
Consider two Laurent polynomials $a(\zeta),b(\zeta)\in\LaurentZ$
with $b(\zeta)$ nonzero and its highest and lowest powers of~$\zeta$
having invertible coefficients;
that is, $b(\zeta)=\zeta^{-\beta}\tilde b(\zeta)$ with $\beta\in\Z$
and $\tilde b(\zeta)\in\Z[\zeta]$ a true polynomial of the form
$\tilde b(\zeta)=\sum_{i=0}^d\tilde b_i\zeta^i$ where $\tilde b_0=\pm1$
and $\tilde b_d=\pm1$.
There exist a unique Laurent polynomial $Q(\zeta)\in\LaurentZ$
and a unique polynomial $R(\zeta)\in\Z[\zeta]$ such that
$$
a(\zeta)=Q(\zeta)b(\zeta)+R(\zeta),\quad\deg R<\deg\tilde b.
$$
\end{proposition}

In practice, our implemented Laurent polynomial division algorithm
starts by scaling the result of the division theorem for the true
polynomials $\tilde a$ and~$\tilde b$ to get an initial remainder of
the form $\zeta^{-\alpha}R(\zeta)$, and then it translates the initial
remainder by a succession of Laurent polynomials $\zeta^{j}c_j\tilde b(\zeta)$
for $j=-\alpha,-\alpha+1,\dotsc,-1$ to eliminate its principal part and
leave a true polynomial remainder of degree less than~$\deg\tilde b$.

\subsection{Construction of a determining truncation bound}

As above, given a weight $k$ and an index~$m$, and given a basic theta block
$\phi(\tau,z)\in\Jcusp k{m'}$ where $0<m'<m$, having baby theta block
$b(\zeta)$, we want to determine an integer $\nmax=\nmax(b)$
such that any weight~$k$, index~$m$ Jacobi cusp form
$g(\tau,z)=\sum_{n\ge1}g_n(\zeta)q^n$ in $\LaurentZ[[q]]$ is
divisible by~$b(\zeta)$ if $g_n(\zeta)$ is divisible by~$b(\zeta)$
for $n=1,\dotsc,\nmax$.

For completeness, we quickly review a standard result.

\begin{lemma}[Valence inequality for subgroups\label{Gammavalence}]
Let $k$ be an integer.  Let $\Gamma$ be a congruence subgroup
of~$\SLtwoZ$ 
with index $I=[\SLtwoZ:\Gamma]$.
For any elliptic modular form $f\in\MFs\Gamma$, having Fourier
expansion $f(\tau)=\sum_{r\in\Qnn}\fc rfq^r$,
$$
\text{$\fc rf=0$ for all $r\le\frac{kI}{12}$}
\implies
f=0.
$$
\end{lemma}

\begin{proof}
Let $\SLtwoZ=\bigsqcup_{i=1}^I\Gamma m_i$.  The symmetrization
${\rm N}f=\prod_{i=1}^If\wtk{m_i}$ lies in $\MFswtgp{Ik}\SLtwoZ$
and its support lies in the sum of the supports of the~$f\wtk{m_i}$,
so
$$
\sum_{i=1}^I\inf{\rm supp}(f\wtk{m_i})
=\inf\bigg(\sum_{i=1}^I{\rm supp}(f\wtk{m_i})\bigg)
\le\inf{\rm supp}({\rm N}f)
$$
Especially if $kI/12<\inf{\rm supp}(f)$
then also $kI/12<\inf{\rm supp}({\rm N}f)$,
and so ${\rm N}f=0$ by the valence inequality for $\SLtwoZ$.
It follows that $f=0$ as desired.
\end{proof}

We return to the question of divisibility.

\begin{proposition}\label{divbyrootsofunity}
Let an integral weight~$k$ and a positive integral index~$m$ be given.
Let $r$ be a positive integer, $s$ an integer, and $\nu$ a
nonnegative integer.
Let $\Gamma_1(r)$ be the usual congruence subgroup of~$\SLtwoZ$
from elliptic modular forms, whose elements satisfy the entrywise
congruence $\smallmatabcd=\smallmat1*01\mymod r$, and introduce
the group
$$
\Gamma(r,m)=\lset\smallmatabcd\in\Gamma_1(r):cm=0\mymod{r^2}\rset.
$$
For any Jacobi form $g\in\Jackm$, we have an equivalence of two
conditions, the second finite:
$$
\text{$(\zeta-\e(s/r))^\nu$ divides the Fourier expansion of~$g$
in $\C[\zeta,\zeta\inv][[q]]$}
$$
if and only if
$$
\sum_{\rho\in\Z}\rho^j\fcJ{n,\rho}g\,\e(\rho s/r)=0,\quad
0\le j<\nu,\ 0\le n\le\tfrac{k+j}{12}\,[\SLtwoZ:\Gamma(r,m)].
$$
\end{proposition}

\begin{proof}
Because the proof will use the translated function $g(\tau,z+s/r)$, a
Jacobi form under a subgroup, we first discuss such Jacobi forms in
general.
Thus consider any subgroup $G$ of~$\PtwooneQ$ commensurable with $\PtwooneZ$.
Consider also any Jacobi form $h\in\Jackm(G)$, by which we mean a
holomorphic function $h:\UHP\times\C\lra\C$ such that the associated
function
$$
\JFtoSMFm h:\UHPtwo\lra\C,\qquad
(\JFtoSMFm h)(\smallpvar)=h(\tau,z)\,\e(m\omega)
$$
transforms under~$G$ as a Siegel modular form of weight~$k$,
and $(\JFtoSMFm h)\wtk\gamma$ is bounded on $\lset\Im\Omega>Y_o\rset$ for
any $\gamma\in\PtwooneQ$ and any positive $2\times2$ real matrix~$Y_o$.

Recall the map $\iota_1:\SLgp(2)\lra\Ptwoone$
from section~\ref{sectionParaDefs}.
We show, following \cite{ez85}, that the leading coefficient of the
Taylor expansion of~$h$ about~$z=0$ is an elliptic modular form
under the group
$$
\Gamma=\lset\gamma\in\SLtwoZ:\iota_1(\gamma)\in G\rset.  
$$
Suppose that this Taylor expansion vanishes to order at least $\nu$,
$$
h(\tau,z)=\sum_{j\ge\nu}\chi_j(h;\tau) z^j.    
$$
From the Jacobi form transformation law
$$
h(\gamma(\tau),z/j(\gamma,\tau))
=j(\gamma,\tau)^k\,\e(mcz^2/j(\gamma,\tau))h(\tau,z),
\quad\gamma=\smallmatabcd\in\Gamma, 
$$
the leading Taylor coefficient of~$h$ transforms as a modular form of
weight $k+\nu$ under~$\Gamma$,
$$
\chi_\nu(h;\gamma(\tau))=j(\gamma,\tau)^{k+\nu}\chi_\nu(h;\tau),
\quad\gamma\in\Gamma.
$$
Also, the leading coefficient is holomorphic at the cusps, inheriting
this property from~$\JFtoSMFm h$, as can been seen from the extended
Cauchy integral formula, and so truly $\chi_\nu(h)\in\MFswtgp{k+\nu}\Gamma$.
This result also follows from Theorem~3.2 in~\cite{ez85}, but the
elementary proof given here is available under our more specific
circumstances.

We show that the Taylor expansion of $h$ at $z=0$ vanishes to
order at least~$\nu$ if and only if  
$$
\sum_{\rho\in\Q}\rho^j\fcJ{n,\rho}h=0,\quad
0\le j<\nu,\ 0\le n\le\tfrac{k+j}{12}\,[\SLtwoZ:\Gamma].
$$
Indeed, the sum in the display is the scaled Fourier coefficient
$j!(2\pi i)^{-j}\fc n{\chi_j(h)}$.
Thus, if $\chi_j(h)=0$ for $j<\nu$ then the sum is~$0$ for such~$j$
and for all~$n\in\Qnn$.
For the other implication, assume the displayed condition.
If $\chi_j(h)\ne0$ for some minimal $j<\nu$, then
$\chi_j(h)\in\MFswtgp{k+j}\Gamma$ by the previous paragraph, and now
the displayed condition and the valence inequality combine to show
that $\chi_j(h)=0$ after all.  Thus the Taylor expansion of~$h$
at~$z=0$ vanishes to order at least~$\nu$.
In a moment, this proof of Proposition~\ref{divbyrootsofunity} will
specialize $G$ and~$\Gamma$ to groups such that the discussion in this
paragraph can take $\rho\in\Z$ and~$n\in\Znn$.

We return to the proposition, in which a Jacobi form
$g\in\Jackm=\Jackm(\Ptwoone(\Z))$ is given.
In fact, the function $(\JFtoSMFm g)(\Omega)=g(\tau,z)\e(m\omega)$
associated to~$g$ is $\wtk{P_m}$-invariant, where $P_m$ is the supergroup
of $\PtwooneZ$ obtained by adjoining the additional generator
$\smallmat1b01$ with $b=\smallmat000{1/m}$.  Let $h(\tau,z)=g(\tau,z+s/r)$,
having associated function $\JFtoSMFm h=(\JFtoSMFm g)\wtk t$ where
$t=\smallmat1b01$ with $b=\smallmat0{s/r}{s/r}0$; thus
$\JFtoSMFm h$ is $\wtk{t\inv P_mt}$-invariant.
The subgroup $G=t\inv P_mt$ of~$\PtwooneQ$ is commensurable with
$\PtwooneZ$, and the subgroup~$\Gamma$ of $\SLtwoZ$ taken by~$\iota_1$
into~$G$ is the group $\Gamma(r,m)$ given in the theorem.
Now, the condition
$$
\text{$(\zeta-\e(s/r))^\nu$ divides the Fourier expansion of~$g$
in $\C[\zeta,\zeta\inv][[q]]$}
$$
holds if and only if
$(z-s/r)^\nu$ correspondingly divides the Taylor expansion
$g(\tau,z)=\sum_j\chi_j(g;\tau)(z-s/r)^j$ in $\C[[q]][[z]]$, and
this is equivalent to the Taylor expansion
$h(\tau,z)=\sum_j\chi_j(g;\tau)z^j$ vanishing to order at least~$\nu$.
Because the Fourier coefficients of~$h$ are
$\fcJ{n,\rho}h=\fcJ{n,\rho}g\e(\rho s/r)$, the previous paragraph
shows that this last condition is
$$
\sum_{\rho\in\Z}\rho^j\fcJ{n,\rho}g\,\e(\rho s/r)=0,\quad
0\le j<\nu,\ 0\le n\le\tfrac{k+j}{12}\,[\SLtwoZ:\Gamma(r,m)].
$$
This completes the proof.
\end{proof}

Recall that basic theta blocks were introduced in section~\ref{sectionThet}.

\begin{theorem}\label{divboundcor}
Let $k$ be an integer and $m$ a positive integer.
Consider a Jacobi cusp form $g\in\Jcusp km$, having Fourier expansion
$$
g(\tau,z)=\sum_n g_n(\zeta)q^n,\quad g_n(\zeta)=\sum_r\fcJ{n,r}g\zeta^r.
$$
Let $m'<m$ be a nonnegative integer.
Consider a Jacobi cusp form basic theta block $\TB\in\Jcusp k{m'}$,
$$
\TB(\tau,z)=\eta(\tau)^{\nu(0)}
(\vartheta(\tau,z)/\eta(\tau))^{\nu(1)}
\prod_{r\ge2}\atom r(\tau,z)^{\nu(r)},
$$
where $\nu:\Znn\lra\Z$ is finitely supported and is nonnegative
on~$\Zpos$, and consider its baby theta block,
$$
b(\zeta)=\zeta^{-\frac12\sum_{r\ge1}\phi(r)\nu(r)}
\prod_{r\ge1}\Phi_r(\zeta)^{\nu(r)}\in\Z[\zeta,\zeta\inv],
$$
where $\phi$ is Euler's totient function and $\Phi_r$ is the $r$-th
cyclotomic polynomial.
Then we have an equivalence of two conditions, the second finite:
$$
\text{$b(\zeta)$ divides the Fourier expansion of~$g$
in~$\C[\zeta,\zeta\inv][[q]]$}
$$
if and only if (with $\Gamma(r,m)$ as in Proposition~\ref{divbyrootsofunity})
$$
\text{$b(\zeta)$ divides $g_n(\zeta)$ in~$\C[\zeta,\zeta\inv]$
for all $n\le\max_{\substack{r\ge1: \,\nu(r)>0}}
\tfrac{k+\nu(r)-1}{12}\,[\SLtwoZ:\Gamma(r,m)]$}.
$$
Furthermore,
these two conditions are implied by a third condition that is
independent of~$m$,
$$
\text{$b(\zeta)$ divides $g_n(\zeta)$ in~$\C[\zeta,\zeta\inv]$
for all $n\le\max_{\substack{r\ge1: \,\nu(r)>0}}
\tfrac{k+\nu(r)-1}{12}r^3\prod_{p\mid r}(1-1/p^2)$},
$$
taking the product over prime divisors of~$r$.
\end{theorem}

\begin{proof}
We prove the nontrivial implication between the first two conditions.
Let $B$ denote the given bound,
$\max_{r\ge1,\nu(r)>0}\tfrac{k+\nu(r)-1}{12}[\SLtwoZ:\Gamma(r,m)]$.
Fix some $r\ge1$ such that $\nu(r)>0$ and some $s\in(\Z/r\Z)^\times$.
Suppose that $(\zeta-\e(s/r))^{\nu(r)}$ divides $g_n(\zeta)$
in~$\C[\zeta,\zeta\inv]$ for all~$n\le B$.
Consequently the Taylor series of $g_n(\zeta)$ at~$z=s/r$ vanishes to
order~$\nu(r)$ in $\C[[z-s/r]]$ for all~$n\le B$.
The $j$-th Taylor coefficient is
$\tfrac{(2\pi i)^j}{j!}\sum_\rho\rho^j\fcJ{n,\rho}g\,\e(\rho s/r)$,
so the vanishing condition implies that
$\sum_\rho\rho^j\fcJ{n,\rho}g\,\e(\rho s/r)=0$  
for all $j<\nu(r)$ and $n\le\tfrac{k+j}{12}[\SLtwoZ:\Gamma(r,m)]$.
Now Proposition~\ref{divbyrootsofunity} says that $(\zeta-\e(s/r))^{\nu(r)}$ 
divides the Fourier expansion of~$g$ in $\C[\zeta,\zeta\inv][[q]]$.
Gathering this result over all $r$ and~$s$ shows that if $b(\zeta)$
divides $g_n(\zeta)$ in~$\C[\zeta,\zeta\inv]$ for all~$n\le B$ then
$b(\zeta)$ divides the Fourier expansion of~$g$ in $\C[\zeta,\zeta\inv][[q]]$.

For the last statement of the theorem, let $\Gamma(r)$ denote the
usual principal congruence subgroup of~$\SLtwoZ$ from elliptic modular
forms, and recall that its index is $r^3\prod_{p\mid r}(1-1/p^2)$.
The containment $\smallmat100r\Gamma(r)\smallmat100r\inv\subset\Gamma(r,m)$
gives the result.
\end{proof}

For example, the basic theta block
$\TB(\tau,z)=\eta(\tau)^{92}(\vartheta_2(\tau,z)/\eta)^2$ lies in
$\Jcusp{46}4$.  To determine whether some~$g$ in $\Jcusp{46}8$
is divisible by the baby theta block $b$ of~$\TB$,
hence by~$\TB$,   
note that
the positive $\nu$-function values of~$\TB$ on positive integers are
$\nu(1)=\nu(2)=2$.  Thus the third condition in
Theorem~\ref{divboundcor} says that it suffices to check the
divisibility of~$g_n(\zeta)$ by~$b(\zeta)$ for~$n\le23$.

\section{Confirming a truncation\label{sectionConf}}

Let $k$ be an integer and  $\JFind$  a positive integer.
Consider a Laurent polynomial
$$
\tilde\psi(\tau,z)=\sum_{n=\nmin}^{\JFind/4}\psi_n(\zeta)q^n,\quad
\text{each }\psi_n(\zeta)\in\LaurentZ.
$$
This section shows that checking whether $\tilde\psi$ truncates any
element of $\Jwh0\JFind$ reduces to checking whether it truncates any
element of $\Jcusp{12i}\JFind/\Delta^i$, where
$\Delta=\eta^{24}\in\Jcusp{12}0$ is the discriminant function from
elliptic modular forms, and $i=i(\tilde\psi)$ is a computable
nonnegative integer.  Granting a suitable basis of $\Jcusp{12i}\JFind$,
the latter check is routine linear algebra.

\subsection{Valuation\label{sectionVal}}

Before addressing the problem, we review the valuation function of a
Jacobi form.  This topic is taken from \cite{gsz}.

\begin{definition}
The Jacobi valuation function
$$
\operatorname{ord}:\Jwh k\JFind\lra{\mathcal C}(\R/\Z)
$$
is
$$
\operatorname{ord}(g;x)=\min_{(n,r)\in{\rm supp}(g)}(n+rx+\JFind x^2),\quad
g\in\Jwh k\JFind,\ x\in\R.
$$
\end{definition}

Each Jacobi valuation function image $\operatorname{ord}(g)$
is known to be $\Z$-periodic,
continuous, and in fact piecewise quadratic.  The valuation takes
products to sums, \ie,
$\operatorname{ord}(g_1g_2)=\operatorname{ord}(g_1)+\operatorname{ord}(g_2)$.
The minimum of the valuation function of a
given~$g$ is denoted $\operatorname{Ord}(g)$ and, for index $\JFind>0$, we have
$$
\operatorname{Ord}(g)
=\min_x\operatorname{ord}(g;x)
=\min_{x,(n,r)\in{\rm supp}(g)}(n+rx+\JFind x^2)
=\min_{(n,r)\in{\rm supp}(g)}\frac{D(n,r)}{4\JFind}\,.
$$
Especially, $g$ is a Jacobi cusp form if and only if $\operatorname{Ord}(g)>0$.
The valuation function of the discriminant function~$\Delta$ is the
constant function~$1$, and so the valuation function of~$\Delta^i$ is
the constant function~$i$ for~$i\in\Znn$.

Our Borcherds product algorithm will compute values $\operatorname{Ord}(\psi)$
for weakly holomorphic Jacobi forms~$\psi$ of weight~$0$ and
index~$N$, with $N$ the level where we seek Borcherds products.
As in Theorem~\ref{wt0integralitythm}, any $\psi\in\Jwh0N$ is
determined by its truncation
$\tilde\psi(\tau,z)=\sum_{n=\nmin}^{N/4}\psi_n(\zeta)q^n$,
and this truncation suffices to compute $\operatorname{Ord}(\psi)$.

\subsection{Confirmation test}

As above, given a Laurent polynomial $\tilde\psi$, we want to test
whether it is the determining truncation of some $\psi\in\Jwh0\JFind$ by
testing instead whether it is the determining truncation of some
$\psi\in\Jcusp{12i}\JFind/\Delta^i$ for a suitable~$i$ that we can compute.
To illustrate our use of the minimum valuation function~$\operatorname{Ord}$, we first
review that
$$
\Jwh0\JFind=\bigcup_{i\in\Znn}\Jcusp{12i}\JFind/\Delta^i
\quad\text{(ascending union)}.
$$
Indeed, the containment $\Delta\Jcusp k\JFind\subset\Jcusp{k+12}\JFind$
gives the chain of containments
$$
0=\Jcusp0\JFind\subset\Jcusp{12}\JFind/\Delta
\subset\Jcusp{24}\JFind/\Delta^2\subset\Jcusp{36}\JFind/\Delta^3\subset\cdots,
$$
and each space $\Jcusp{12i}\JFind/\Delta^i$ lies in~$\Jwh0\JFind$ 
%
%
because~$\Delta$ is nonzero on~$\UHP$ and $\Delta^{-i}$ has finite principal part.  
To see that the union is all of~$\Jwh0\JFind$,
note that for every $g\in\Jwh0\JFind$ there is some $i\in\Znn$ such
that the valuation function
$\operatorname{Ord}(\Delta^ig)=i+\operatorname{Ord}(g)$ is positive;
thus $\Delta^ig$ lies in~$\Jcusp{12i}\JFind$, and so
$g\in\Jcusp{12i}\JFind/\Delta^i$.  
Furthermore, this argument shows that $g\in\Jcusp{12i}\JFind/\Delta^i$ as
soon as $i>-\operatorname{Ord}(g)$.

\medskip

Now it is clear how to test whether a given Laurent polynomial
truncates some weakly holomorphic Jacobi form of weight~$0$ and a
given index~$\JFind$.

\begin{proposition}\label{testLaurentTrunc}
Let a positive integer index~$\JFind$ be given, and let a Laurent
polynomial be given as follows,
$$
\tilde\psi(\tau,z)=\sum_{n=\nmin}^{\JFind/4}\psi_n(\zeta)q^n,\quad
\text{each }\psi_n(\zeta)\in\LaurentZ,\quad\psi_\nmin(\zeta)\ne0.
$$
Consider any $i\in\Znn$ such that $i>-D(n,r)/(4\JFind)$ for all
$(n,r)\in{\rm supp}(\tilde\psi)$.
Compute truncations of $\Jcusp{12i}\JFind$ basis elements
$g(\tau,z)\in q\LaurentZ[[q]]$,
$$
\tilde g(\tau,z)=q\sum_{n=0}^{\JFind/4+i-1}g_{1+n}(\zeta)q^n,\quad
\text{each }g_{1+n}(\zeta)\in\LaurentZ,
$$
and compute the corresponding truncation of~$\Delta^i$,
$$
\widetilde{\Delta^i}(\tau)=q^i\sum_{n=0}^{\JFind/4+i-1}\Delta_{i,i+n}q^n,\quad
\Delta_{i,i}=1,\text{ each }\Delta_{i,i+n}\in\Z,
$$
and compute their quotients to the same accuracy,
$$
(\widetilde{g/\Delta^i})(\tau,z)
=\sum_{n=1-i}^{\JFind/4}g'_n(\zeta)q^n,\quad
\text{each }g'_n(\zeta)\in\LaurentZ.
$$
Then $\tilde\psi$ is the truncation of some $\psi\in\Jwh0\JFind$ if and only if
$\tilde\psi$ lies in the $\Q$-span of the truncations $\widetilde{g/\Delta^i}$.
\end{proposition}

In the proposition we could instead compute the product                           
$\tilde\psi\,\widetilde{\Delta^i}$ to the same accuracy and check
whether it lies in the $\Q$-span of the truncations $\tilde g$.
However, the proposition is laid out to facilitate testing many
Laurent polynomials~$\tilde\psi$ for a given~$i$, computing the
quotients $\widetilde{g/\Delta^i}$ once each rather than the product
$\tilde\psi\widetilde\Delta^i$ for every~$\tilde\psi$.

\begin{proof}
($\implies$) Suppose that the given Laurent polynomial $\tilde\psi$
truncates some $\psi\in\Jwh0\JFind$.  Then
$$
\min_{(n,r)\in{\rm supp}(\tilde\psi)}D(n,r)
=\min_{(n,r)\in{\rm supp}(\psi)}D(n,r),
$$
and so $\psi\in\Jcusp{12i}\JFind/\Delta^i$ for the smallest $i\in\Znn$ such
that $i>-D(n,r)/(4\JFind)$ for all $(n,r)\in{\rm supp}(\tilde\psi)$.
Because $\psi_\nmin(\zeta)\ne0$, it follows that $i\ge1-\nmin$,
and so the truncations $\widetilde{g/\Delta^i}$ in the proposition
encompass a sum from~$\nmin$ to~$\JFind/4$.
The $\widetilde{g/\Delta^i}$ are determining truncations of a basis of
the space $\Jcusp{12i}\JFind/\Delta^i$ over~$\C$.  Because the given
Laurent polynomial $\tilde\psi$ has coefficients in~$\LaurentZ$,
Theorem~\ref{wt0integralitythm} says that it lies in the $\Q$-linear
span of the basis.

($\impliedby$) This is clear, because the $\widetilde{g/\Delta^i}$ are
determining truncations of a basis of the subspace
$\Jcusp{12i}\JFind/\Delta^i$ of~$\Jwh0\JFind$.
\end{proof}

\subsection{Computational confirmation\label{subsecAlgo}}


We have three methods to check in practice whether a Laurent polynomial
$\tilde\psi(\tau,z)=\sum_{n=\nmin}^{N/4}\psi_n(\zeta)q^n$ with
each $\psi_n(\zeta)\in\LaurentZ$ and $\psi_\nmin(\zeta)\ne0$ truncates
some element~$\psi$ of $\JkNcusp$.  
The first method is a complete algorithmic solution, with a speed-up in a particular case. 
The second and third methods are ``lucky searches,'' \ie, fast attempts that could be run first 
even though they do not always succeed.  

The first method is the test described in Proposition~\ref{testLaurentTrunc},
which will find any match that exists provided that the computations
are tractable.  
This method confirmed all the candidate truncations for
$\psi\in\JzeroNwh$ that gave rise to weight~$4$ symmetric Borcherds
products in \cite{psysqfree}, with $(\BPc,t)=(1,0)$ in all cases.  A
disadvantage of this method is that it describes $\psi$ using a long
linear combination of $\Jcusp{12i}N$ basis elements, whose coefficients
are rational numbers with enormous numerators and denominators.
For $(\BPc,t)=(1,1)$, we can improve this method by subtracting
$\widetilde{(\phi_1|V_2)/\phi_1}=q\inv-G_1(\zeta)+\cdots$
from $\tilde\psi(\tau,z)=q\inv-G(\zeta)+\cdots$ to get an initial
expansion from~$q^0$ to~$q^{\JFind/4}$, and then checking whether the
difference truncates an element of $\Jcusp{12}N/\Delta$; if so then
that element takes the form $\psi-(\phi_1|V_2)/\phi_1$ where
$\tilde\psi$ truncates~$\psi$.

The second method is particular to~$t=0$.  Consider all basic theta blocks
$\phi\in\JkNcusp$ with $\BPA=1$, \ie,
$\phi(\tau,z)=q\,b(\zeta)(1-G(\zeta)q+\cdots)$.  For each~$\phi$ 
satisfying the condition~$\nu(2r) \le \nu(r)$ of Theorem~\ref{whquotienttheorem},  
the quotient
$(\phi|V_2)/\phi$ lies in~$\JzeroNwh(\Z)$ and has principal part~$0$; that
is, each quotient takes the form $G(\zeta)+\bigOh q$ where $G(\zeta)$
is the germ of the theta block.  Search the space spanned by these
quotients for an element that truncates to~$\tilde\psi$, which also has 
principal part~$0$.  This method can exhibit $\psi$ as a short linear  
combination of elements, and the coefficients have tended to be small
integers in the cases calculated so far, but we make no
general assertions here.
Failure of this method does not preclude $\tilde\psi$ truncating
some~$\psi\in\Jwh0N$.

For the third method, the context from the algorithm is that we have a
basic theta block $\phi\in\JkBPcNcusp$, and we are seeking
$\psi=g/\phi$ where $g\in\Jcusp k{(\BPc+1)N}$ is divisible by the
baby theta block $b(\zeta)$ of~$\phi$.  As a special case of
this, a basic theta block $\Theta\in\Jcusp k{(\BPc+1)N}$ that is divisible
by~$b(\zeta)$ is called an {\em inflation\/} of~$\phi$, and we
can produce inflations~$\Theta$ of~$\phi$ and then search 
for~$\psi$ in the space spanned by the quotients $\Theta/\phi$.
This method with $(\BPc,t)=(2,0)$ produced weight~$2$ nonlifts of
levels $N=277,249,295$, needing only one inflation~$\Theta$ each time.
That is, in these cases, not only is the leading Fourier--Jacobi
coefficient~$\phi_2\in\Jcusp2{2N}$ of the Borcherds product $\BL(\psi)$ a
basic theta block, but so is the next Fourier--Jacobi
coefficient~$\phi_3\in\Jcusp2{3N}$.
Again, failure of this method does not preclude $\tilde\psi$
truncating some~$\psi\in\Jwh0N$.

\medskip

We give an example of the second method, set in $\SwtKlevel2{277}$ 
and mentioned in the introduction.  
Letting $0^e$ abbreviate $\eta^e$ and letting $r^e$ abbreviate
$(\vartheta_r/\eta)^e$ for~$r\ge1$, consider the three basic theta blocks
$\phi_1=0^4 1^2 2^2 3^2 4^1 5^1 14^1 17^1$,
$\phi_2=0^4 1^1 3^1 4^2 5^1 6^1 8^1 9^2 15^1$,
$\phi_3=0^4 1^1 2^1 3^1 4^2 5^1 7^1 8^1 9^1 17^1$,
all in $\Jcusp2{277}$, and consider the linear combination
$\psi=-(\phi_1|V_2)/\phi_1-(\phi_2|V_2)/\phi_2+(\phi_3|V_2)/\phi_3$,
in $\Jwh0{277}(\Z)$.
Because the $q^0$-coefficient of each $-(\phi|V_2)/\phi$ is the germ
$G(\phi)$, the $q^0$-coefficient of~$\psi$ is
the germ $G(\phi_1)+G(\phi_2)-G(\phi_3)$ of the theta
block $\phi_4=0^4 1^2 2^1 3^2 4^1 5^1 6^1 7^{-1} 9^1 14^1 15^1$, a basic
theta block with denominator.
The linear combination $\psi$ has nonnegative Humbert multiplicities,
and so $\BL(\psi)$ lies in $\SwtKlevel2{277}$ and its 
leading Fourier--Jacobi coefficient is~$\phi_4$.  
Its lowest $q$-power and $\xi$-power are $q^A=q^1$ and
$\xi^C=\xi^{277}$, so that our algorithm parameters
$(c,t)=(C/277,A-c)$ are~$(1,0)$.
This Borcherds product is not a Gritsenko lift.
The existence of a nonlift dimension in $\SwtKlevel2{277}$
was shown in \cite{py15}, but the nonlift construction here
is new.  We will discuss more nonlift constructions in
\cite{pyinflation}.

We give an example of the third method, set in $\SwtKlevel2{249}$.
Consider the two basic theta blocks
$\Theta=0^4 1^2 2^2 3^2 4^2 5^2 6^2 7^1 8^2 9^1 10^1 11^1 12^1 13^1
14^1 18^1\in\Jcusp2{3\cdot249}$ and
$\phi=0^4 1^3 2^2 3^2 4^2 5^2 6^3 7^2 8^1 9^1 10^1 11^1 12^1 13^1
\in\Jcusp2{2\cdot249}$.
The baby theta block quotient $b(\Theta)/b(\phi)$ is the germ $G(\phi)$.
The quotient $\psi=\Theta/\phi$ is
$(\vartheta_8\vartheta_{18}\vartheta_{14})/
(\vartheta_1\vartheta_6\vartheta_7)$,
an element of $\Jwh0{249}(\Z)$.
It has nonnegative Humbert multiplicities, and so $\BL(\psi)$ lies in
$\SwtKlevel2{249}$.
Its leading Fourier--Jacobi coefficient is~$\phi$.
Its lowest $q$-power and $\xi$-power are $q^A=q^2$ and
$\xi^C=\xi^{498}$, so that our algorithm parameters
$(c,t)=(C/249,A-c)$ are~$(2,0)$.
This Borcherds product is not a Gritsenko lift.
Its source-form $\psi$ was described as a quotient of
$\vartheta$-functions in \cite{psysqfree}, but here we explain how we
found it.

We give an example combining the second and third methods, set in
$\SwtKlevel2{587}$.  Consider the basic theta blocks
$\phi=0^4 1^2 2^3 3^2 4^2 5^2 6^2 7^1 8^2 9^1 10^1 11^1 12^1 13^1 14^1$
in $\Jcusp2{587}$
and $\Xi=0^4 1^1 2^2 3^2 4^2 5^1 6^2 7^1 8^1 9^1 10^2 12^1 13^1 14^1
15^1 16^1 18^1 22^1$ in $\Jcusp2{2\cdot587}$.
Their quotient $\Xi/\phi$ is
$(\vartheta_{10}\vartheta_{18}\vartheta_{15}\vartheta_{16}\vartheta_{22})/
(\vartheta_{1}\vartheta_{3}\vartheta_{5}\vartheta_{8}\vartheta_{11})$,
an element of $\Jwh0{587}(\Z)$.
Let $\psi=(\phi|V_2)/\phi-\Xi/\phi$.
Calculations show that the $q^0$-coefficient of~$\psi$ is the germ
$G(\phi)$, and that $\psi$ has nonnegative Humbert multiplicities;
hence $\BL(\psi)$ lies in $\SwtKlevel2{587}^-$.
The parameters $(c,t)$ here are $(1,1)$.
This example is from \cite{gpy16}.

\section{Cuspidality criterion\label{sectionCusp}}

We establish a test to determine whether a
paramodular form is a cusp form.

\begin{theorem}\label{cuspinesstheorem}
Let an integer weight~$k$ and a positive level~$N$ be given.
Associate to each positive divisor $m$ of~$N$ the quantities
$\ell=\gcd(N/m,m)$ and $\delta$ such that $N/m=\delta\ell$.  
Introduce a subgroup of~$\SLtwoZ$ associated with~$\ell$,
$$
\widetilde\Gamma_{\!1}(\ell)=\langle\Gamma_{\!1}(\ell),-I_2\rangle
=\begin{cases}
\Gamma_{\!1}(\ell)&\text{if $\ell=1,2$},\\
\Gamma_{\!1}(\ell)\sqcup-\Gamma_{\!1}(\ell)&\text{if $\ell\ge3$},
\end{cases}
$$
whose index $\widetilde I_1(\ell)=[\SLtwoZ:\widetilde\Gamma_{\!1}(\ell)]$ is
$$
\widetilde I_1(\ell)=\begin{cases}
1&\text{if $\ell=1$},\\
3&\text{if $\ell=2$},\\
\tfrac12\ell^2\prod_{p\mid\ell}(1-1/p^2)&\text{if $\ell\ge3$}.
\end{cases}
$$
A paramodular form $f\in\MkKN$ is a cusp form if and only if 
the following finitely many Fourier coefficients vanish: 
$$
\fc{n\delta\smallmat{\phantom{-}1}{-m}{-m}{\phantom{-}m^2}}f=0,\quad
0<m\mid N,\ 
0\le n\le k\widetilde I_1(\ell)/12.
$$
Furthermore, if $k$ is odd then the vanishing condition in the
previous display needs to be checked only when~$\ell\ge3$.
\end{theorem}

\begin{proof}
Because the coefficient indices in the theorem are all singular, only
the ``if'' needs proof.

By definition, a paramodular form $f\in\MkKN$ is a cusp form if
$\Phi(f\wtk g)=0$ for all $g\in\SptwoQ$.  Introduce the matrices
$$
\alpha_m=\smallmat1m01,\quad
\gamma_m=\smallmat{\alpha_m}00{\alpha_m^*}\qquad
\text{for any positive divisor~$m$ of~$N$}.
$$
H.~Reefschl\"ager's decomposition (\cite{reefschlager73}, and see
Theorem~1.2 of~\cite{py13})
$$
\SptwoQ=\bigsqcup_{0<m\mid N}\KN\gamma_m\PtwooneQ
$$
shows for any $g\in\SptwoQ$ that
$\Phi(f\wtk g)=\Phi(f\wtk{\gamma_m}\wtk u)$
for some $0<m\mid N$ and $u\in\PtwooneQ$.
Let $u_1\in\GLtwoQ$ denote the $2\times2$ matrix of upper left
entries of the four blocks of~$u$.
(Thus the map $\iota_1$ from section~\ref{sectionParaDefs} is a
section of the map here from matrices~$u$ to matrices~$u_1$.)
One can check that the upper left entry of $u(\smallmat\tau00\omega)$
is $u_1(\tau)$ and its lower right entry goes to~$i\infty$ as $\omega$
does so, and also $j(u,\smallmat\tau00\omega)=d_{22}j(u_1,\tau)$
where $d_{22}\in\Qx$ is the $(4,4)$-entry of~$u$.
Thus for any $\tau\in\UHP$,
letting $\noway(\tau,\omega)$ denote the lower right entry
of~$u(\smallmat\tau00\omega)$,
\begin{align*}
\big(\Phi(f\wtk{\gamma_m}\wtk{u})\big)(\tau)
&=\lim_{\omega\to i\infty}
(f\wtk{\gamma_m}\wtk{u})(\smallmat\tau00\omega)\\
&=d_{22}^{-k}j(u_1,\tau)^{-k}\lim_{\omega\to i\infty}
(f\wtk{\gamma_m})(\smallmat{u_1(\tau)}**{\noway(\tau,\omega)})\\   
&=d_{22}^{-k}j(u_1,\tau)^{-k}
\big(\Phi(f\wtk{\gamma_m})\big)(u_1(\tau))\\
&=d_{22}^{-k}\big(\Phi(f\wtk{\gamma_m})\wtk{u_1}\big)(\tau),
\end{align*}
which is to say,
\begin{equation}\label{thm8.1tag}
\Phi(f\wtk{\gamma_m}\wtk u)
=d_{22}^{-k}\Phi(f\wtk{\gamma_m})\wtk{u_1}.
\end{equation}
It follows that $f$ is a cusp form when
$\Phi(f\wtk{\gamma_m})=0$ for each positive divisor~$m$ of~$N$.
Equivalently, $f$ is a cusp form when for each positive divisor~$m$
of~$N$, the function
$$
f_m=(\Phi(f\wtk{\gamma_m}))\wtk{\smallmat100\delta}
$$
is zero.  We will study the function~$f_m$.

We show that $f_m=0$ if $k$ is odd and $\ell\in\{1,2\}$.
The condition on~$\ell$ says that $2$ lies in the ideal generated by
$N/m$ and~$m$, and so $N\mid (2m-\lambda m^2)$ for some integer~$\lambda$.
The computation
$$
\rmat1m01\rmat{-1}0\lambda1\rmat1{-m}01=\cmat*{2m-\lambda m^2}**
\in\GampmsupzeroN
$$
shows that $\smallmat{-1}0{\phantom{-}\lambda}1\matoplus
\smallmat{-1}0{\phantom{-}\lambda}1^*$ lies in
$\gamma_m\inv\KN\gamma_m$.
It follows that
$$
f\wtk{\gamma_m}
\wtk{\smallmat{-1}0{\phantom{-}\lambda}1\matoplus
\smallmat{-1}0{\phantom{-}\lambda}1^*}
=f\wtk{\gamma_m}.
$$
Apply Siegel's $\Phi$-map to the left side, and note that the upper
left entries of the four blocks of the $\PtwooneQ$ matrix
$\smallmat{-1}0{\phantom{-}\lambda}1\matoplus
\smallmat{-1}0{\phantom{-}\lambda}1^*$
make up the negative identity matrix~$-I_2$, to get by~\eqref{thm8.1tag}
in the previous paragraph
$$
\Phi(f\wtk{\gamma_m}
\wtk{\smallmat{-1}0{\phantom{-}\lambda}1\matoplus
\smallmat{-1}0{\phantom{-}\lambda}1^*})
=\Phi(f\wtk{\gamma_m})\wtk{-I_2},
$$
which is $(-1)^k\Phi(f\wtk{\gamma_m})$.
On the other hand, apply Siegel's $\Phi$-map to the right side of
the penultimate display to get $\Phi(f\wtk{\gamma_m})$.
Because $k$ is odd and these are equal, $\Phi(f\wtk{\gamma_m})=0$,
and consequently $f_m=0$ as claimed.

We explain that Theorem~4.3 of \cite{py13} and its proof show that for
each positive divisor~$m$ of~$N$, the function $f_m$ lies in
$\MFs{\widetilde\Gamma_{\!1}(\ell),\chi^k}$, where the character
$\chi:\widetilde\Gamma_{\!1}(\ell)\lra\{\pm1\}$ takes the value~$1$
on~$\Gamma_{\!1}(\ell)$ and for $\ell\ge3$ it takes the value~$-1$
on~$-\Gamma_{\!1}(\ell)$.
By the previous paragraph, we may assume that $k$ is even or~$\ell\ge3$.
Theorem~4.3 of \cite{py13} states that taking each matrix~$u$ in the
group $\gamma_m\inv\KN\gamma_m\cap\PtwooneQ$ and forming from it the
corresponding matrix~$u_1=\smallmatabcd$ of upper left entries of the
four blocks of~$u$ produces the group
$\smallmat100\delta\widetilde\Gamma_{\!1}(\ell)\smallmat100\delta\inv$.
Further, the proof shows that given $u_1=\smallmatabcd$ in the latter group,
any $u$ that gives rise to it has $(4,4)$-entry $d_{22}=\pm1$ such
that $d_{22}=d\mymod\ell$.
Because $k$ is even or~$\ell\ge3$, this determines~$d_{22}^k$.
Given a paramodular form $f\in\MkKN$, consider any
$\tilde u_1\in\widetilde\Gamma_{\!1}(\ell)$ and let
$u_1=\smallmat100\delta\tilde u_1\smallmat100\delta\inv$,
which arises from some $u\in\gamma_m\inv\KN\gamma_m\cap\PtwooneQ$,
and compute
\begin{alignat*}2
f_m\wtk{\tilde u_1}
&=\big(\Phi(f\wtk{\gamma_m})\big)\wtk{\smallmat100\delta}\wtk{\tilde u_1}
&&\quad\text{by definition of~$f_m$}\\
&=\big(\Phi(f\wtk{\gamma_m})\big)\wtk{u_1}\wtk{\smallmat100\delta}
&&\quad\text{by the relation between $\tilde u_1$ and~$u_1$}\\
&=d_{22}^{k}\Phi(f\wtk{\gamma_m}\wtk u)\wtk{\smallmat100\delta}
&&\quad\text{by~\eqref{thm8.1tag} in the penultimate paragraph}\\
&=d_{22}^{k}\Phi(f\wtk{\gamma_m})\wtk{\smallmat100\delta}
&&\quad\text{because $u\in\gamma_m\inv\KN\gamma_m$}\\
&=d_{22}^{k}f_m
&&\quad\text{by definition of~$f_m$}.
\end{alignat*}
Because $\tilde u_1$ and~$u_1$ have the same diagonal elements, the
value $d_{22}=\pm1$ satisfies $d_{22}=d\mymod\ell$ where now $d$ is
the lower right entry of~$\tilde u_1$; thus $d_{22}^k=\chi^k(\tilde u_1)$.
The remaining thing to show is that $f_m$ is analytic at the cusps of
$\widetilde\Gamma_{\!1}(\ell)$, which is to say that given any
$g=\smallmatabcd\in\SLtwoZ$, the function $f_m\wtk g$ is analytic
at~$i\infty$.
Let $u_1=\smallmat100\delta g\smallmat100\delta\inv
=\smallmat a{b/\delta}{c\delta}d\in\SLtwoQ$,
and let
$$
u=\iota_1(u_1)=\left[\begin{array}{cc|cc}
a & 0 & b/\delta & 0 \\
0 & 1 & 0 & 0 \\
\hline
c\delta & 0 & d & 0 \\
0 & 0 & 0 & 1
\end{array}\right]\in\PtwooneQ,
$$
so $f_m\wtk g=\Phi(f\wtk{\gamma_m})\wtk{\smallmat100\delta}\wtk{g}
=\Phi(f\wtk{\gamma_m})\wtk{u_1}\wtk{\smallmat100\delta}
=\Phi(f\wtk{\gamma_m u})\wtk{\smallmat100\delta}$.
By the K\"ocher Principle, $f\wtk{\gamma_m u}$ is bounded on the set
$\lset\Im \Omega>Y_o\rset$ for every positive $2\times2$ real
matrix~$Y_o$, and this makes $f_m\wtk g$ analytic at~$i\infty$, as
desired.  This completes the proof that $f_m$ lies in
$\MFs{\widetilde\Gamma_{\!1}(\ell),\chi^k}$.

In consequence of $f_m$ lying in $\MFs{\widetilde\Gamma_{\!1}(\ell),\chi^k}$
and the character $\chi^k$ being trivial or quadratic, $f_m^2$ lies
in~$\MFswtgp{2k}{\widetilde\Gamma_{\!1}(\ell)}$, and so it is the zero
function when its Fourier coefficients of index at most
$2k\widetilde I_1(\ell)/12$ vanish.  Equivalently, $f_m=0$ when
its Fourier coefficients of index at most $k\widetilde I_1(\ell)/12$ vanish.
To compute the Fourier coefficients of~$f_m$, let the given
paramodular form $f$ have Fourier expansion
$$
f(\Omega)=\sum_{t\in\XtwoNsemi}\fc tf\,\e(\ip t\Omega).
$$
Because $\det \alpha_m=1$, $f\wtk{\gamma_m}(\Omega)$ is simply
$f(\Omega[\alpha_m'])$, and because
$\ip t{\Omega[\alpha_m']}=\ip{t[\alpha_m]}\Omega$,
the Fourier expansion of~$f$ gives
$$
f\wtk{\gamma_m}(\Omega)
=\sum_{t\in\XtwoNsemi}\fc tf\,\e(\ip{t[\alpha_m]}\Omega).
$$
Each $t=\smallmat n{r/2}{r/2}{\mu N}\in\XtwoNsemi$ gives rise to
$t[\alpha_m]=\smallmat n{mn+r/2}{mn+r/2}{m^2n+mr+\mu N}$.  Thus the
image of $f\wtk{\gamma_m}$ under Siegel's $\Phi$-map sums only
over indices~$t$ such that $r=-2mn$ to make the off-diagonal entries
of $t[\alpha_m]$ zero and $\mu N=m^2n$ to make the lower right entry
zero as well.  Because $N/m=\ell\delta$ and $m=\ell\delta'$ with
$\gcd(\delta,\delta')=1$, the condition $\mu N=m^2n$ gives $\delta\mid n$.
Now we have
$$
\Phi\big(f\wtk{\gamma_m}\big)(q)
=\sum_{n\ge0,\ \delta\mid n}
\fc{n\smallmat{\phantom{-}1}{-m}{-m}{\phantom{-}m^2}}f\,q^n,
$$
and applying $\wtk{\smallmat100\delta}$ gives the Fourier expansion of~$f_m$,
$$
f_m(q)=\sum_{n\ge0}\fc n{f_m}\,q^n,\qquad
\fc n{f_m}
=\delta^{-k}\fc{n\delta\smallmat{\phantom{-}1}{-m}{-m}{\phantom{-}m^2}}f
\text{ for each~$n$}.
$$

Altogether, $f$ is a cusp form when for each positive divisor~$m$
of~$N$, the Fourier coefficients $\fc n{f_m}
=\delta^{-k}\fc{n\delta\smallmat{\phantom{-}1}{-m}{-m}{\phantom{-}m^2}}f$
vanish for $0\le n\le k\widetilde I_1(\ell)/12$.  This proves the
first statement in the theorem.
The second statement follows from the fact that $f_m=0$ if
$k$ is odd and $\ell\in\{1,2\}$.
\end{proof}

When using Theorem~\ref{cuspinesstheorem} to check whether a Borcherds
product is a cusp form, we exploit the fact that the Borcherds product~$f$ is
either symmetric or antisymmetric, so that
$$
\fc\smallpindN f=0\iff\fc\smallpindnN f=0.
$$
If~$n<m$ then the latter Fourier coefficient in the previous display
is easier to compute for a Borcherds product.
In getting a Borcherds product Fourier expansion from the weakly
holomorphic Jacobi form that gives rise to the Borcherds product, we are
getting the Jacobi coefficients one at a time and at greater expense
with each Jacobi coefficient, so the difficulty is measured by the
size of~$m$ in the index $\smallpindN$.  By contrast, when getting the
Fourier expansion of a Gritsenko lift or the product of Gritsenko
lifts, the difficulty is measured by the size of $4nmN-r^2$.

\section{Borcherds product algorithm\label{sectionBLAlgorithm}}

This section presents our algorithm to find all holomorphic Borcherds
products of a given weight~$k$ and level~$N$ that have a specified $A$ and~$C$.  
%
%
That is, the
algorithm finds all $\BL(\psi)\in\MkKN$ of the form
$$
\BL(\psi)(\Omega)
=q^A\xi^Cb(\zeta)(1-G(\zeta)q+\cdots)\exp(-\Grit(\psi)(\Omega)).
$$
Here $G(\zeta)$ is the $q^0$-coefficient of~$\psi(\tau,z)$.
Noting, by the involution conditions, that $A \ge C/N$, 
we write $A$ and~$C$ in terms of two other parameters, $c$ and $t$, 
in the following format, 
$$
(A,C)=(\BPc+t,\BPc N),\quad\BPc\in\Zpos,\ t\in\Znn.
$$
We call~$t$ the {\em offset\/}.  
To justify our claim of finding {\em all \/}Borcherds products in $\SkKN$, we note that $c$ is bounded above by
the number of Fourier-Jacobi coefficients that determine  $\SkKN$;  this bound is determined  
in \cite{bpy16} and some improvements in special cases can be found in \cite{psysqfree}.  
Also, for each~$c$ there are only a finite number of~$t$ to consider because there are only a 
finite number of theta blocks in $\JkBPcNcusp$.  
Before laying out the algorithm, we show that
$$
\text{the Borcherds product $f=\BL(\psi)$ is}\
\begin{cases}
\text{symmetric}&\text{if~$t=0$ or~$t=2$},\\
\text{antisymmetric}&\text{if~$t=1$}.
\end{cases}
$$
For an offset~$t\ge3$, symmetric and antisymmetric Borcherds products are
possible.
Further we show that, decomposing the leading theta block
$\phi\in\JkBPcNcusp$ of~$f$ as
$$
\phi(\tau,z)=q^{\BPc+t}b(\zeta)(1-G(\zeta)q+\cdots),
$$
the source-form~$\psi\in\JzeroNwh$ of~$f$ is
$$
\psi(\tau,z)=\begin{cases}
G(\zeta)+\sum_{n\ge1}\psi_n(\zeta)q^n&\text{if $t=0$},\\
q\inv+G(\zeta)+\sum_{n\ge1}\psi_n(\zeta)q^n&\text{if $t=1$},\\
\sum_{n=1-t}^{-1}\psi_n(\zeta)q^n+G(\zeta)
+\sum_{n\ge1}\psi_n(\zeta)q^n&\text{if $t\ge2$},
\end{cases}
$$
which has $t$ principal terms for $t=0,1$ and at most $t-1$ principal
terms for~$t\ge2$.  After this, we describe our algorithm to find all
paramodular Borcherds products for given $k$, $N$, $\BPc$, and~$t$.

\subsection{Offset $\boldsymbol{t=0}$: simple symmetric
  case\label{sectionACmmN}}

Take $(A,C)=(\BPc,\BPc N)$.  The leading Fourier--Jacobi coefficient of
$f=\BL(\psi)$ is a basic theta block in~$\JkBPcNcusp$,
$$
\phi_\BPc(\tau,z)=q^\BPc b(\zeta)(1-G(\zeta)q+\cdots),
$$
giving
$$
F_{\BPc,\BPc}(\zeta)=b(\zeta),\qquad
F_{\BPc+1,\BPc}(\zeta)=-b(\zeta)G(\zeta).
$$
From the involution condition $F_{\BPc,\BPc+1}=(-1)^k\epsilon F_{\BPc+1,\BPc}$, the
nonzero portion of the table of $q^n\xi^{mN}$-coefficients
$F_{n,m}(\zeta)$ of~$f$ begins
$$
\begin{array}{|c||c|c|}
\hline
F_{n,m}(\zeta) & n=\BPc                            & n=\BPc+1   \\
\hline\hline
m=\BPc         & b(\zeta)                          & -b(\zeta)G(\zeta) \\
\hline
m=\BPc+1       & -(-1)^k\epsilon\,b(\zeta)G(\zeta) & *           \\
\hline
\end{array}
$$
and so the next Fourier--Jacobi coefficient of~$f$ is
$$
\phi_{\BPc+1}(\tau,z)
=q^\BPc b(\zeta)\left(-(-1)^k\epsilon G(\zeta)+\cdots\right).
$$
Thus the quotient $\psi=-\phi_{\BPc+1}/\phi_\BPc$ has principal
part~$0$,
$$
\psi(\tau,z)=\frac{(-1)^k\epsilon G(\zeta)+\bigOh q}{1+\bigOh q}
=(-1)^k\epsilon G(\zeta)+\bigOh q.
$$
Because $G(\zeta)$ is the $q^0$-coefficient of~$\psi(\tau,z)$,
the previous display gives $(-1)^k\epsilon=1$,
making $f$ symmetric.  Alternatively, because $\psi$ has principal
part~$0$, the quantity~$D_0$ in Theorem~\ref{BPthm} is~$0$, and again
$f$ is symmetric.  Either argument gives 
$$
\psi(\tau,z)=G(\zeta)+\bigOh q.
$$

\subsection{Offset $\boldsymbol{t=1}$: simple antisymmetric
  case\label{sectionACm+1mN}}

Take $(A,C)=(\BPc+1,\BPc N)$.  The leading Fourier--Jacobi coefficient
of $f=\BL(\psi)$ is
$$
\phi_\BPc(\tau,z)=q^{\BPc+1}b(\zeta)(1-G(\zeta)q+\cdots),
$$
and so the nonzero portion of the $F_{n,m}$-table for~$f$ begins
$$
\begin{array}{|c||c|c|c|}
\hline
F_{n,m}(\zeta)  & n=\BPc                   & n=\BPc+1  & n=\BPc+2          \\
\hline\hline
m=\BPc          & 0                        & b(\zeta)  & -b(\zeta)G(\zeta) \\
\hline
m=\BPc+1        & (-1)^k\epsilon\,b(\zeta) & *         & *                 \\
\hline
\end{array}
$$
making the next Fourier--Jacobi coefficient of~$f$
$$
\phi_{\BPc+1}(\tau,z)=q^\BPc b(\zeta)\left((-1)^k\epsilon+\cdots\right).
$$
The quotient $\psi=-\phi_{\BPc+1}/\phi_\BPc$ has one principal term,
as can be seen using the leading terms of the numerator and denominator,
$$
\psi(\tau,z)
=\frac1q\,\cdot\,\frac{-(-1)^k\epsilon+\bigOh q}{1+\bigOh q}
=-(-1)^k\epsilon q\inv+\bigOh1.
$$
Thus the quantity~$D_0$ in the Borcherds product theorem
is~$-(-1)^k\epsilon=\pm1$, which is odd, and so $f$ is antisymmetric,
giving $(-1)^k\epsilon=-1$.
Again, $G(\zeta)$ is the $q^0$-coefficient of~$\psi(\tau,z)$, and so
$$
\psi(\tau,z)=q\inv+G(\zeta)+\bigOh q.
$$
Alternatively, we can see this by noting that now
$F_{\BPc+1,\BPc+1}(z)=0$, and so using the first two terms
of~$-\phi_{\BPc+1}$ and~$\phi_\BPc$ gives
$$
\psi(\tau,z)
=\frac1q\,\cdot\,\frac{1+\bigOh{q^2}}{1-G(\zeta)q+\bigOh{q^2}}
=q\inv(1+\bigOh{q^2})(1+G(\zeta)q+\bigOh{q^2}),
$$
which again is $\psi(\tau,z)=q\inv+G(\zeta)+\bigOh q$.

\subsection{Offset~$\boldsymbol{t=2}$: second symmetric case}

Take $(A,C)=(c+2,cN)$.
We show that in this case, any Borcherds product~$f=\BL(\psi)$ is
symmetric.  The nonzero portion of the $F_{n,m}$-table begins
$$
\begin{array}{|c||c|c|c|}
\hline
F_{n,m}(\zeta)  & n=\BPc & n=\BPc+1 & n=\BPc+2  \\
\hline\hline
m=\BPc          & 0      & 0        & b(\zeta) \\
\hline
m=\BPc+1        & 0      & *        & *        \\
\hline
\end{array}
$$
If $f$ is antisymmetric then the coefficient function
$F_{\BPc+1,\BPc+1}(\zeta)$ is~$0$, making the quotient
$\psi=-\phi_{\BPc+1}/\phi_\BPc$ have principal part~$0$, giving the
contradiction that $f$ is symmetric after all.

\subsection{Offset~$\boldsymbol{t\ge2}$: shortened principal
  part\label{sectionRCO>=2}}

Take $(A,C)=(\BPc+t,\BPc N)$ with $t\ge2$.  The nonzero portion of the
$F_{n,m}$-table for any relevant Borcherds product $f=\BL(\psi)$ begins, 
by the involution conditions $F_{m,n}=(-1)^k\epsilon\,F_{n,m}$,  
$$
\begin{array}{|c||c|c|c|c|c|}
\hline
F_{n,m}(\zeta)  & n=\BPc & n=\BPc+1 & \dots & n=\BPc+t-1 & n=\BPc+t \\
\hline\hline
m=\BPc          & 0      & 0        & \dots & 0          & b(\zeta) \\
\hline
m=\BPc+1        & 0      & *        & \dots & *          & *        \\
\hline
\end{array}
$$
Thus the lowest possible power of~$q$ in the quotient
$\psi=-\phi_{\BPc+1}/\phi_\BPc$ is $q^{1-t}$ rather than $q^{-t}$ as
was the case for~$t=0,1$, and $\psi$ is determined by its truncation
$$
\tilde\psi(\tau,z)=\sum_{n=1-t}^{N/4}\psi_n(\zeta)q^n,\quad
\psi_0(\zeta)=G(\zeta).
$$

\subsection{Algorithm}

Let $f=\BL(\psi)$ be a Borcherds product paramodular cusp form
arising from a weakly holomorphic Jacobi form~$\psi\in\Jwh0N$.
As above, $(A,C)=(\BPc+t,\BPc N)$ for some $\BPc\in\Zpos$ and $t\in\Znn$.
In light of Theorem~\ref{wt0integralitythm} and of the principal part
calculations just above, $\psi$ is determined by a truncation that
takes the form
$$
\tilde\psi(\tau,z)=\begin{cases}
G(\zeta)+\sum_{n=1}^{N/4}\psi_n(\zeta)q^n&\text{if $t=0$},\\
q\inv+G(\zeta)+\sum_{n=1}^{N/4}\psi_n(\zeta)q^n&\text{if $t=1$},\\
\sum_{n=1-t}^{-1}\psi_n(\zeta)q^n+G(\zeta)
+\sum_{n=1}^{N/4}\psi_n(\zeta)q^n&\text{if $t\ge2$}.
\end{cases}
$$
The algorithm to follow finds all such truncations that could arise
from a suitable~$\psi$ and then checks which such truncations actually  
do so.  The truncations used in the algorithm can be longer than those
in the previous display.
Each weakly holomorphic Jacobi form~$\psi\in\Jwh0N$ that we seek is
{\em singular-integral\/}, by which we mean that its Fourier
coefficients $\fcJ{n,r}\psi$ for $4nN-r^2\le0$ lie in~$\Z$.  As in
Theorem~\ref{wt0integralitythm}, $\psi$ is determined by its singular
Fourier coefficients.

To find the truncation of a putative $\psi$ in the Borcherds product theorem,
let $f=\BL(\psi)$ and let $\phi\in\JkBPcNcusp$ be the leading theta
block of~$f$,
$$
\phi(\tau,z)=q^{\BPc+t}b(\zeta)(1-G(\zeta)q+\cdots).
$$
Let $\delta=0$ if $t=0,1$, and let $\delta=1$ if~$t\ge2$.
Because the product $g=\psi\phi$ is the additive inverse of
the Fourier--Jacobi coefficient $\phi_{\BPc+1}$ of~$f$, it lies
in~$\Jcusp k{(\BPc+1)N}$ and its Laurent expansion has the form
$$
g(\tau,z)=\sum_{n\ge\BPc+\delta}g_n(\zeta)q^n,
\quad\text{$b(\zeta)\mid g_n(\zeta)$ in~$\LaurentQ$ for all~$n$}.
$$
Suitable truncations of $g$ and~$\phi$,
\begin{align*}
\tilde g(\tau,z)&=q^{\BPc+\delta}\sum_{n=0}^{N/4+t-\delta}
g_{\BPc+\delta+n}(\zeta)q^n\\
\tilde\phi(\tau,z)&=q^{\BPc+t}\sum_{n=0}^{N/4+t-\delta}
\phi_{\BPc+t+n}(\zeta)q^n,
\end{align*}
have a calculable truncated quotient containing the same number of terms,
$$
\tilde\psi(\tau,z)=\frac{\tilde g(\tau,z)}{\tilde\phi(\tau,z)}\mymod q^{N/4+1},
\qquad
\tilde\psi(\tau,z)=\sum_{n=\delta-t}^{N/4}\psi_n(\zeta)q^n,
$$
and this is the desired determining truncation of~$\psi=g/\phi$.

Our algorithm begins with some general initiation steps, assuming that
we have at hand long-enough $\LaurentZ[q]$ truncations of a basis
of $\Jcusp k{(\BPc+1)N}$ whose elements lie in $\LaurentZ[[q]]$.
After the initiation, for each basic theta block $\phi\in\JkBPcNcusp$
that could arise from the Borcherds product theorem as the leading theta
block of $\BL(\psi)$ for some~$\psi$, the algorithm creates all
Laurent polynomials
$\tilde\psi(\tau,z)=\sum_{n=\delta-t}^{N/4+\nextra}\psi_n(\zeta)q^n$
that could truncate such a~$\psi$.  Finally it determines which of these
Laurent polynomials really are such truncations.
The algorithm-parameter $\nextra$ provides additional truncation
length that can serve three purposes:
\begin{itemize}
\item It decreases the chance of apparent but false multiples of baby
  theta blocks in step~4 below, or even guarantees no such false
  multiples by using Theorem~\ref{divboundcor}.
\item It lets us check cuspidality in step~9 below by using
  Theorem~\ref{cuspinesstheorem}.
\item It produces longer truncations of the source-form~$\psi$ of the
  Borcherds product, so as to find more Fourier coefficients of the
  Borcherds product itself.
\end{itemize}
Because the algorithm uses $\nextra$ for this mixture of purposes, it
admits small refinements that involve decomposing $\nextra$ into
parts, but we give the simple version here for clarity.

\medskip

\noindent{\bf Step~0.}  If the offset~$t$ is~$0$ or~$1$ then set
$\delta=0$; otherwise set $\delta=1$.

\smallskip

\noindent{\bf Step~1.}  Find a maximal linearly independent set of
initial $\Jcusp k{(\BPc+1)N}$ expansions in $\LaurentZ[q]$, each
having the form
$$
\tilde g(\tau,z)=\sum_{n=1}^{N/4+\BPc+t+\nextra}g_n(\zeta)q^n.
$$

\smallskip

\noindent{\bf Step~2.}  In the $\Z$-module spanned by the initial
expansions~$\tilde g$ in step~1, find a maximal linearly independent
set of expansions in $\LaurentZ[q]$ whose first $\BPc+\delta-1$
coefficients are~$0$ (reusing the symbol~$\tilde g$ here),
$$
\tilde g(\tau,z)=q^{\BPc+\delta}\sum_{n=0}^{N/4+t-\delta+\nextra}
g_{\BPc+\delta+n}(\zeta)q^n.
$$

\smallskip

\noindent{\bf Step~3.}  Select a basic theta block $\phi\in\JkBPcNcusp$ of
the form
$$
\phi(\tau,z)=b(\zeta)q^{\BPc+t}(1-G(\zeta)q+\cdots),
$$
with the sum $1-G(\zeta)q+\cdots$ stored to~$q^{N/4+t-\delta+\nextra}$.
Section~\ref{sectionFind} explained how to find all basic theta blocks.  
Thus the truncation of the basic theta block is
$$
\tilde\phi(\tau,z)=q^{\BPc+t}b(\zeta)
\sum_{n=0}^{N/4+t-\delta+\nextra}\phi_{\BPc+t+n}(\zeta)q^n.
$$
The remaining steps of this algorithm depend on~$\phi$.  After
completing them, we return to this step and select another~$\phi$.
When none remain, the algorithm is done.

\smallskip

\noindent{\bf Step~4.}  In the $\Z$-module spanned by the initial
expansions in step~2, find a maximal linearly independent set of
expansions whose coefficients are divisible in $\LaurentZ$ by the baby
theta block $b(\zeta)$ of~$\phi(\tau,z)$,
$$
\tilde g_\phi(\tau,z)=q^{\BPc+\delta}\sum_{n=0}^{N/4+t-\delta+\nextra}
g_{\phi,\BPc+\delta+n}(\zeta)q^n,\quad
\text{$b(\zeta)\mid g_{\phi,\BPc+\delta+n}(\zeta)$ for each~$n$}.
$$
The corresponding nontruncated elements $g_\phi(\tau,z)$ of
$\Jcusp k{(\BPc+1)N}$ are not guaranteed to be divisible
by~$b(\zeta)$ unless $\nextra$ is large enough for Theorem~\ref{divboundcor}
to apply, though even with smaller~$\nextra$ they may well be.

\smallskip

\noindent{\bf Step~5.}  Divide each~$\tilde g_\phi(\tau,z)$ from
step~4 by the truncated basic theta block $\tilde\phi(\tau,z)$ from step~3
to get linearly independent elements of $q^{\delta-t}\LaurentZ[q]$,
$$
h(\tau,z)=\sum_{n=\delta-t}^{N/4+\nextra}h_n(\zeta)q^n.
$$
Each such~$h$ projects to the vector space $\bigoplus_{n,r}\Q q^n\zeta^r$,
where the sum is taken only over pairs $n,r$ such that $4nN-r^2\le0$
and $-N\le r<N$, \ie, over singular index classes.  Each
projection~$w_h$ has integral coordinates.  Check whether these
projections are independent.

If so, then compute a saturating integral basis of their $\Q$-span,
meaning a basis whose integral linear combinations are all the
integral elements of their $\Q$-span.
Each basis vector is a rational linear combination of the
vectors~$w_h$.  Replace the truncations in the previous display by
their corresponding rational linear combinations, and now let $h$
denote any of these new truncations.  The new truncations~$h$
are singular-integral (again, this means that $\fcJ{n,r}h\in\Z$
for $4nN-r^2\le0$) but overall they now lie in $q^{\delta-t}\LaurentQ[q]$,
no longer necessarily in $q^{\delta-t}\LaurentZ[q]$.
Let $H$ denote the $\Z$-module generated by the new truncations~$h$.

If the projections $w_h$ are not independent then step~4 has produced
some truncations~$\tilde g_\phi$ such that $g_\phi$ isn't divisible
by~$b(\zeta)$, and so the process of creating~$H$ could lose some of
the good information created in step~4.  Thus the algorithm could
miss some source-forms $\psi$ of Borcherds products.  The algorithm
aborts and reports this.  We should try again with a larger~$\nextra$.  

This step calls for two remarks.
First, dependence of the projections $w_h$ need not flag all instances
when step~4 has produced truncations of non-multiples of~$b(\zeta)$.
Such dependence does flag all instances when bad truncations could
make the algorithm miss Borcherds products.
In step~9 below, the algorithm can recognize another potential problem
stemming from bad truncations---an integer linear programming problem
might have infinitely many solutions---and abort.
Also, the algorithm might create false candidate truncations of
Borcherds product source-forms~$\psi$ in consequence of bad truncations,
but it is guaranteed to diagnose their falsity and discard them in
step~9 below.
The second remark about this step is that by Theorem~\ref{wt0integralitythm},
we could instead check for dependence among projections to the smaller
space arising from pairs $(n,r)$ such that $4nN-r^2<0$ and $-N\le r<N$,
possibly catching more cases where step~4 has produced bad truncations.
However, we need the integrality of the $4nN-r^2=0$ coefficients as well,
so still we would compute a saturating integral basis of the projection
onto all the singular coordinates, described at the beginning of this step.
Further, false truncations are not necessarily fatal to the algorithm
so long as the projection onto all the singular coordinates is
injective.  Thus, the algorithm can eschew the extra check, or carry
out the extra check and abort in the case of false truncations, or
continue in the case of false truncations but with a flag set.

\smallskip

\noindent{\bf Step~6.}  Search for an element $\tilde\psi_o\in H$
whose constant term is $G(\zeta)$,
$$
\tilde\psi_o(\tau,z)=\sum_{n=\delta-t}^{N/4+\nextra}\psi_{o,n}(\zeta)q^n,
\qquad
\psi_{o,0}=G.
$$

\smallskip

\noindent{\bf Step~7.}  In the $\Z$-module $H$ spanned by the
vectors~$h$ in step~5, find a maximal linearly independent set of
expansions whose constant coefficients are~$0$,
$$
h_o(\tau,z)=\sum_{n=\delta-t}^{N/4+\nextra}h_{o,n}(\zeta)q^n,
\qquad h_{o,0}=0.
$$
As in step~5, each such~$h_o$ projects to the vector space
$\bigoplus_{n,r}\Q q^n\zeta^r$, taking the sum only over singular
index classes.  The projections are independent.
Compute a saturating integral basis of their $\Q$-span.
Replace the truncations in the previous display by their corresponding
rational linear combinations, and let $H_o$ denote the $\Z$-module
that they generate.  Let $h_{o,i}$ for $i=1,\dotsc,d$
denote these new truncations, with $d=\dim(H_o)$.  The~$h_{o,i}$
are singular-integral.
If desired, carry out some further reduction on the basis elements
$h_{o,i}$ to simplify them, \eg, by reducing the sizes of their
singular coefficients.

\smallskip

\noindent{\bf Step~8.}  We have a lattice-translate of candidate
truncations,
$$
\tilde\psi\in\tilde\psi_o+H_o.
$$
Find the candidates~$\tilde\psi$ for which all Humbert surface
multiplicities in the divisor ${\rm div}(\BL(\psi))$ are nonnegative,
as follows.  From section~\ref{sectionBPTh}, the relevant Humbert
surfaces are parametrized by pairs $(D,r)$ where $D=4nN-r^2<0$ and
$|r|\le N$ and $\fcJ{j^2n,jr}{\tilde\psi}\ne0$ for some~$j\in\Zpos$.
Run through all pairs $(n,r)$ with $\delta-t\le n<N/4$ and
$4nN<r^2\le N^2$.  For each pair, check whether any Fourier coefficient
$\fcJ{j^2n,jr}{\tilde\psi_o}$ or $\fcJ{j^2n,jr}{h_{o,i}}$ is nonzero;
this requires checking only those~$j$ such that $j^2(4nN-r^2)$ is at
least the minimum conceivable discriminant $4(\delta-t)N-N^2$ where
the Jacobi form is supported.  If the check finds a nonzero
coefficient then add $(n,r)$ to the list of Humbert surface
parameter-pairs where checking is needed.

Once all the pairs to check have been determined, form the matrix~$M$
whose rows are indexed by the ``need to check'' parameter-pairs
$(n,r)$ and whose columns are indexed by $i=1,\dotsc,d$ where
$d=\dim(H_o)$, and whose $(n,r)\times i$ entry is
$\sum_{j\ge1}\fcJ{j^2n,jr}{h_{o,i}}$.  Also form the column vector~$b$
whose rows are indexed by the same parameter-pairs $(n,r)$, and whose
$(n,r)$-entry is $\sum_{j\ge1}\fcJ{j^2n,jr}{\tilde\psi_o}$.
The entries of $M$ and~$b$ are integers.  We seek integer column
vectors~$x$, indexed by~$i$, such that $Mx+b\ge0$ entrywise.  This is
an integer linear programming problem;  
we use an integer linear programming module that accepts input $M$, $b$, and $s \in \Zpos$ 
and guarantees the output of all integral~$x$ such that $Mx+b \ge 0$, 
 if the number of these solutions~$x$ is less than~$s$.  
Each solution~$x$ determines a
candidate truncation $\tilde\psi=\tilde\psi_o+\sum_{i=1}^d x_ih_{o,i}$
of a Jacobi form $\psi\in\JzeroNwh$ such that $\BL(\psi)$ lies in~$\MkKN$.
If step~4 has produced only truncations~$\tilde g_\phi$ of~$g_\phi$
that are multiples of~$b(\zeta)$, hence of~$\phi$,
in the ring of holomorphic functions, then the integer linear programming
problem $Mx+b\ge0$ has only finitely many solutions by
Theorem~\ref{finitelymanyBPthm}.  But if step~4 has produced some
truncations~$\tilde g_\phi$ such that $g_\phi$ isn't divisible
by~$b(\zeta)$ then the problem $Mx+b\ge0$ could have infinitely many
solutions.  Thus, solve the problem, seeking at most~$s$ solutions,
where $s$ is some large value, to ensure that the process terminates;
if fewer than~$s$ solutions are returned then they are all the solutions.

\smallskip

\noindent{\bf Step~9.}  For each candidate~$\tilde\psi$, try to find  
an actual weakly holomorphic Jacobi form $\psi\in\JzeroNwh$ that
truncates to~$\tilde\psi$.  We have three computational methods to do
so, as described in section~\ref{subsecAlgo}; alternatively, if
$\nextra$ is large enough for Theorem~\ref{divboundcor} to apply
then the existence of~$\psi$ is guaranteed by Lemma~\ref{newlemmafour} and we may skip confirming it.
If we find such a~$\psi$, or if we know that it exists, then we have a
Borcherds product of the desired $(c,t)$ type, whose leading theta block
is~$\phi$.  We may check its cuspidality using
Theorem~\ref{cuspinesstheorem} if $\nextra$ is large enough.

If step~8 found fewer than $s$ solutions of the integer linear
programming problem $Mx+b\ge0$, then it has not missed any candidates
and the algorithm has performed correctly for the current basic theta block.
Proceed to step~10.

If step~8 found $s$ solutions of the problem $Mx+b\ge0$, and no $\psi$
exists for some candidate~$\tilde\psi$, then step~4 has produced a
truncation of a non-multiple of~$b(\zeta)$, and this may have created
infinitely many solutions of the problem.  Abort, and rerun the
algorithm with a larger value of~$\nextra$.

The remaining case is that step~8 found $s$ solutions of the
problem $Mx+b\ge0$, and some $\psi$ exists for each candidate~$\tilde\psi$.
In this case, we don't know whether step~4 has produced a truncation
of a non-multiple of~$b(\zeta)$, nor whether step~9 has found all
solutions of the problem.  Increase~$s$ and return to solving the
problem in step~8 with this larger cap on the number of solutions.
Eventually the process will land us in one of the other two cases of
this step: the problem has finitely many solutions, or a
candidate~$\tilde\psi$ has no~$\psi$.  Either way, the algorithm moves on.

\smallskip

\noindent{\bf Step~10.} If any basic theta blocks $\phi\in\JkBPcNcusp$
remain for the algorithm then return to step~$3$.  Otherwise terminate.

\subsection{Implementation issues\label{sectionImpl}}

This section briefly discusses implementation aspects of three parts
of the algorithm: Jacobi cusp form bases, division, and saturating an
integral basis.

\medskip

For Jacobi cusp form bases, a premise of the algorithm is that we have
determining truncations of $\Jcusp k{(\BPc+1)N}$ basis elements whose
coefficients are integral Laurent polynomials, the basis elements being
$$
g(\tau,z)=\sum_{n\ge1}g_n(\zeta)q^n,\quad
g_n(\zeta)\in\LaurentZ\text{ for all~$n$}.
$$
We produce truncations of such elements by working with basic theta blocks
without denominator in $\Jcusp k{(\BPc+1)N}$, and with basic theta blocks
without denominator in spaces $\Jcusp k{d(\BPc+1)N}$ followed by an
index-lowering Hecke operator $W_d$ \cite{ks89} that takes them into
$\Jcusp k{(\BPc+1)N}$.  We created such bases on demand rather than
building a systematic database of bases, because making such a basis
can be expensive.

\medskip

We turn to division.
Let $R$ be an integral domain.  The units of the Laurent series ring
$L=R[q\inv][[q]]$ are the Laurent series $b(q)=q^\beta\sum_{n\ge0}b_nq^n$
with $b_0$ a unit in~$R$.  Given any nonzero $a(q)\in L$ and any
invertible $b(q)\in L$, we can compute any specified number $\nmax+1$
terms of the quotient $a(q)/b(q)$ by truncating $a(q)$ and~$b(q)$ to
that many terms and carrying out the corresponding Laurent polynomial
division.  That is, writing $a(q)=q^\alpha\sum_{n\ge0}a_nq^n$, the
quotient $c(q)=a(q)/b(q)$ has leading term $q^{\alpha-\beta}$, and its
coefficients are determined in succession by the relations
$$
a_n=b_nc_0+b_{n-1}c_1+\cdots+b_1c_{n-1}+b_0c_n,\quad n=0,1,\dotsc,
$$
and determining $c_0,\dotsc,c_\nmax$ requires only
$a_0,\dotsc,a_\nmax$ and $b_0,\dotsc,b_\nmax$.

For example, our algorithm divides elements of $\LaurentZ[q]$, all of
whose coefficient functions are known to be divisible by a baby theta
block~$b(\zeta)$,
$$
\tilde g(\tau,z)=q^\BPc\sum_{n=0}^\nmax g_{\BPc+n}(\zeta)q^n,\quad
b(\zeta)\mid g_{\BPc+n}(\zeta)\text{ in }\LaurentZ
\text{ for $n=0,\dotsc,\nmax$}
$$
by a truncation of a basic theta block having the specified baby theta block,
$$
\tilde\phi(\tau,z)
=q^{\BPc+t}b(\zeta)\sum_{n=0}^\nmax\phi_{\BPc+t+n}(\zeta)q^n,
\quad\phi_{\BPc+t}(\zeta)=1,
$$
to get an element of $q^{-t}\LaurentZ[q]$,
$$
h(\tau,z)=\sum_{n=-t}^{\nmax-t}h_n(\zeta)q^n.
$$
To carry out such a division, first divide every coefficient function
$g_n(\zeta)$ by~$b(\zeta)$, confirming that the remainders are~$0$;
the cost of this check is insignificant in comparison to other parts
of our computations.  From here the division is carried out as just above.

The process of dividing a basis of $\Jcusp{12i}N$ by a power~$\Delta^i$
of the discriminant function, as in section~\ref{sectionConf}, is similar.

\medskip

We discuss saturating an integral basis.
Let $n$ be a positive integer, and let the vectors
$v_1,\dotsc,v_d$ in~$\Z^n$ be linearly independent over~$\Z$ and hence
over~$\Q$.
The vector space $V=\bigoplus_{i=1}^d\Q v_i$ contains the
integer lattice $\bigoplus_{i=1}^d\Z v_i$, but this integer lattice
need not be all of the so-called saturated lattice $V(\Z)=V\cap\Z^n$.
To compute an integral basis of~$V$ whose $\Z$-span is all of~$V(\Z)$,
proceed as follows.
\begin{itemize}
\item Let $M$ be the $d\times n$ integer matrix having rows~$v_i$.
\item Let $A\in\GL d\Z$ and $B\in\GL n\Z$ be such that the matrix
  $M_o=AMB$ has integer diagonal entries and all other entries~$0$.
  For example, $M_o$ could be the Smith normal form of~$M$, but we do
  not need the Smith normal form condition that the diagonal entries
  are the elementary divisors that describe the structure
  of~$\bigoplus_i\Z v_i$ as a subgroup of~$\Z^n$.
  Alternatively, $A$ and~$B$ and~$M_o$ can be obtained by repeatedly
  left-multiplying $M$ into Hermite normal form and then transposing
  it, until it is diagonal.
\item Let $w_1,\dotsc,w_d$ denote the first $d$ rows of~$B\inv$.
  These vectors represent the desired basis, \ie,
  $V=\bigoplus_{i=1}^d\Q w_i$ and $V(\Z)=\bigoplus_{i=1}^d\Z w_i$.
\end{itemize}

\section{Examples\label{sectionExam}}

This section gives three more examples of using our algorithm.
To find all paramodular cusp form Borcherds products
for a given weight~$k$ and level~$N$, we need to determine all pairs
$(c,t)$ for which $\Jcusp k{cN}$ can contain basic theta blocks having
lowest $q$-power~$q^{c+t}$.
For $N\le5$, the conditions given at the end of section~\ref{sectionFind}
constrain the possible pairs~$(c,t)$ to a finite quadrilateral.
For squarefree~$N$ we can use an integral closure argument
\cite{py15,psysqfree} to get an upper bound of possible $c$-values,
and for general~$N$ can use the Fourier coefficient bound from \cite{bpy16} to do the
same, and then we get an upper bound of~$t$ for each~$c$ by analyzing
a Jacobi form basis.
These methods determined the pairs $(c,t)$ in the three examples to
follow.

\subsection{Weight~$\boldsymbol{2}$, level~$\boldsymbol{249}$}
This example arises from the para\-modular conjecture.
The space $\SwtKlevel2{249}$ has $6$ dimensions, spanned by Fricke
plus forms, while $\Jcusp2{249}$ is $5$-dimensional and so there is
one nonlift dimension.  The only element of $\SwtKlevel2{249}$
divisible by~$\xi^{249\cdot3}$ is~$0$ (cf.~\cite{psysqfree}), and so
$c\le2$ for all Borcherds products.
The only element of $\Jcusp2{249}$ whose first term $g_1(\zeta)q$
vanishes is~$0$, so every basic theta block in $\Jcusp2{249}$
has~$\BPA\le1$; also, the only element of $\Jcusp2{2\cdot249}$
whose terms $g_1(\zeta)q$ and~$g_2(\zeta)q^2$ both vanish is~$0$,
so every basic theta block in $\Jcusp2{2\cdot249}$ has~$\BPA\le2$.
Thus the only possible leading theta blocks of a Borcherds product in
$\MwtKlevel2{249}$ arise from $(\BPc,t)=(1,0),(2,0)$.
Figure~\ref{tbtable} gives the resulting basic theta blocks, and it shows
how many Borcherds products result from each.
In this figure and in the two figures to follow, an initial entry
$(n_1,n_2)$ in the cell at row~$\BPc$, column~$t$ gives the numbers of basic
theta blocks without denominator and properly with denominator
in $\Jcusp k{\BPc N}$ having leading $q$-power $q^{\BPc+t}$.
The cell then lists the relevant basic theta blocks, with the symbol
$d_1^{e_1}d_2^{e_2}\cdots$ denoting the basic theta block
$\phi=\eta^{2k}\prod_i(\vartheta_{d_i}/\eta)^{e_i}$,
and each basic theta block is followed by the number of Borcherds product
paramodular cusp forms that it gave rise to, out of the total number
of Borcherds product paramodular forms that it gave rise to if there were
noncusp forms as well, and then the $i$-values for locating the source
form~$\psi$ in $\Jwh{12i}N/\Delta^i$.
In Figure~\ref{tbtable}, the Borcherds product arising from the basic theta
block with denominator for $c=1$ and the Borcherds product arising for
$c=2$ are nonlifts, and in the six $c=1$ cases where a basic theta
block gives rise to two Borcherds products, one of them is a nonlift; the
other ten Borcherds products in the figure are Gritsenko lifts.
See \cite{yps17} for detailed descriptions of the Borcherds products
referenced in the three figures of this section.

\begin{figure}[htb]
$$
\begin{array}{|c||l|}
\hline
\rule{0pt}{1.75em}
\begin{aligned}
&\phi\in\Jcusp2{c\cdot249}\\
&q^{c+t}\,\|\,\phi
\end{aligned}
& \ t=0 \\
\hline\hline
\rule{0pt}{9.7em}
\BPc=1 &
\begin{aligned}
&(10,1)\to17\mathcal S\\
&1^1 2^1 3^1 5^1 6^1 7^2 9^1 10^1 12^1:2\mathcal S,i=1\,1\\
&1^1 3^2 5^1 6^3 9^1 11^1 12^1:2\mathcal S,i=1\,1\\
&2^2 3^1 5^2 6^1 7^1 9^1 11^1 12^1:2\mathcal S,i=1\,1\\
&1^2 3^1 4^1 5^1 6^1 8^1 9^1 11^1 12^1:2\mathcal S,i=1\,1\\
&1^1 2^1 3^1 4^1 5^2 7^1 9^1 12^2:1\mathcal S,i=1\\
&2^1 3^2 4^1 5^1 6^1 7^1 9^1 10^1 13^1:2\mathcal S,i=1\,1\\
&1^2 3^1 4^2 5^1 6^1 9^1 12^1 13^1:2\mathcal S,i=1\,1\\
&1^1 2^1 3^2 4^1 5^1 6^1 9^1 11^1 14^1:1\mathcal S,i=1\\
&1^1 2^2 3^1 5^1 6^1 7^1 8^1 9^1 15^1:1\mathcal S,i=1\\
&1^2 2^1 3^1 4^1 5^1 6^1 9^1 10^1 15^1:1\mathcal S,i=1\\
&1^1 2^1 3^2 4^{-1} 5^1 6^1 8^2 9^1 10^1 11^1:1\mathcal S,i=1
\end{aligned}
\\
\hline
\rule{0pt}{1.75em}
\BPc=2 &
\begin{aligned}
&(1,0)\to1\mathcal S\\
&1^3 2^2 3^2 4^2 5^2 6^3 7^2 8^1 9^1 10^1 11^1 12^1 13^1:1\mathcal S,i=1
\end{aligned}\\
\hline
\end{array}
$$
\caption{Basic theta blocks and cusp Borcherds products: weight~$2$,
  level~$249$\label{tbtable}}
\end{figure}

\subsection{Weight~$\boldsymbol{9}$, level~$\boldsymbol{16}$}
This example arose in searching for paramodular forms whose
automorphic representations have supercuspidal components \cite{psysuper}.
The space $\SwtKlevel9{16}$ has $16$ dimensions, with $15$ spanned by
symmetric forms and $1$ by an antisymmetric form.  Because the weight
is odd, the symmetric forms are Fricke minus forms, and 
the antisymmetric form is a Fricke plus form.  
The only possible leading theta blocks of Borcherds products
in $\SwtKlevel9{16}$ are as shown in Figure~\ref{tbtable2}.
Confirming all but one of the candidate truncations $\tilde\psi$
required only $i=1$, using the ``subtraction trick'' described in
section~\ref{subsecAlgo} for the lone $(\BPc,t)=(1,1)$ case.
The truncations in $\Jwh9{(c+1)16}$ were taken to~$q^{16/4+c+t}$, with
$\nextra=0$; there was no need for longer truncations to make the
algorithm run successfully.
The rank of the space spanned by the $14$ symmetric cusp Borcherds
products generated as described in the first column of the table is~$9$.
We know that the table gives all the cusp Borcherds products for this
weight and level because Jacobi restriction (out to the rigorous bound
of $18$ Jacobi coefficients---see Table~3 in \cite{psysuper}), shows
that any element $\sum_{m\ge4}\phi_m(\tau,z)\xi^{m\cdot16}$ of
$\SwtKlevel9{16}$ is~$0$;
then an analysis of the bases of $\Jcusp9{c\cdot16}$ for $c=1,2,3$
finds only the basic theta blocks shown in the table.

\begin{figure}[htb]
$$
\begin{array}{|c||l|l|}
\hline
\begin{aligned}
\rule{0pt}{1.0em}
&\phi\in\Jcusp9{c\cdot16}\\
&q^{c+t}\,\|\,\phi
\end{aligned}
& \ t=0 & \ t=1 \\
\hline\hline
\rule{0pt}{2.75em}
\BPc=1 &
\begin{aligned}
&(0,2)\to7\mathcal S/10\mathcal M\\
&1^{-5} 2^7 3^1:3\mathcal S/6\mathcal M,i=1\,1\,2\,1\,1\,1\\
&1^{-1} 2^2 3^1 4^1:4\mathcal S,i=1\,1\,1\,1\\
\end{aligned}
&
\begin{aligned}
&(1,0)\to1\mathcal S\\
&1^{11} 2^3 3^1:1\mathcal S,i=2\to1\\
&\phantom{1^{-1} 2^2 3^1 4^1}\\
\end{aligned}
\\
\hline
\rule{0pt}{5.80em}
\BPc=2 &
\begin{aligned}
&(6,0)\to6\mathcal S/8\mathcal M\\
&1^2 2^{11} 3^2:1\mathcal S,i=1\\
&1^7 2^3 3^5:1\mathcal S/2\mathcal M,i=1\,1\\
&1^6 2^6 3^2 4^1:2\mathcal S,i=1\,1\\
&1^{10}2^1 3^2 4^2:1\mathcal S,i=1\\
&1^9 2^3 3^2 5^1:1\mathcal S/2\mathcal M,i=1\,1\\
&1^{11} 2^2 3^1 6^1:\emptyset\\
\end{aligned}
&
\begin{aligned}
&(1,0)\to\emptyset\\
&1^{18} 2^7 3^2:\emptyset\\
&\phantom{1^7 2^3 3^5:1\mathcal S/}\\
&\phantom{1^6 2^6 3^2 4^1:2\mathcal S}\\
&\phantom{1^{10}2^1 3^2 4^2:1\mathcal S}\\
&\phantom{1^9 2^3 3^2 5^1:1\mathcal S}\\
&\phantom{1^{11} 2^2 3^1 6^1:0}\\
\end{aligned}\\
\hline
\rule{0pt}{5.0em}
\BPc=3 &
\begin{aligned}
&(5,0)\to1\mathcal S\\
&1^{13} 2^{10} 3^3 4^1:1\mathcal S,i=1\\
&1^{14} 2^7 3^6:\emptyset\\
&1^{17} 2^5 3^3 4^2:\emptyset\\
&1^{16} 2^7 3^3 5^1:\emptyset\\
&1^{18}2^63^26^1:\emptyset\\
\end{aligned}\\
\cline{1-2}
\end{array}
$$
\caption{Basic theta blocks and cusp Borcherds products: weight~$9$,
  level~$16$\label{tbtable2}}
\end{figure}

\subsection{Weight $\boldsymbol{46}$, level $\boldsymbol{4}$\label{kN=464}}
Here we focus on the case $(c,t)=(1,3)$, to illustrate offset $t=3$.
Symmetric and antisymmetric Borcherds product both arise for this
$(c,t)$, in fact arising from the same basic theta block.
Indeed, there is only one basic theta block in $\Jcusp k{cN}=\Jcusp{46}4$
with leading $q$-power~$q^{c+t}=q^4$, and it is
$\phi=\eta^{92}(\vartheta_2/\eta)^2$.

Let $b=b(\zeta)$ denote the baby theta block of $\phi=\phi(\tau,z)$.
Experimentation shows that possibly finding the dimension of
multiples of~$b$ in $\Jcusp{k}{(c+1)N}=\Jcusp{46}8$ requires
$\nextra=5$, thus taking initial expansions to~$q^{N/4+c+t+\nextra}=q^{10}$.
Indeed, for expansions to~$q^{10}$, the algorithm posits $14$
dimensions of multiples of~$b$, whereas for expansions to~$q^9$ it
incorrectly posits~$15$; the~$15$ is recognizably incorrect because
it connotes a resulting $15$-dimensional subspace of $\Jwh04$ that
violates Theorem~\ref{wt0integralitythm} because the singular
coefficients have rank only~$14$.
The prognosis of $14$ dimensions by using expansions to~$q^{10}$ is not
guaranteed to be correct, but running the algorithm with such
expansions produced only correct candidates~$\tilde\psi$ at the end.
In contrast with~$q^{10}$, Theorem~\ref{divboundcor} guarantees 
divisibility by the baby theta block after initial expansions to~$q^{23}$, as noted above
after the theorem's proof.
The basic theta block~$\phi$ gives rise to five Borcherds products, one
symmetric, \ie, in $\SwtKlevel{46}4^+$, and the other four
antisymmetric, \ie, in $\SwtKlevel{46}4^-$.  Confirming the relevant
truncations~$\tilde\psi$ found~$\psi$ in $\Jcusp{24}4/\Delta^2$
for three of the five Borcherds products and in $\Jcusp{36}4/\Delta^3$
for the other two.

Extending the computation to determine by Theorem~\ref{cuspinesstheorem} 
that all five Borcherds products are cuspidal required a much higher
value of~$\nextra$.  Indeed, the theorem with $k=46$ and~$N=4$ gives
$(m,\ell,\delta,\widetilde I,\nmax)$-tuples $(1,1,4,1,3)$,
$(2,2,1,3,11)$, and $(4,1,1,1,3)$.
The three resulting Fourier coefficient indices
$\nmax\delta\smallmat{\phantom{-}1}{-m}{-m}{\phantom{-}m^2}$ are
$\smallmat{\phantom{-}12}{-12}{-12}{\phantom{-}3\cdot4}$,
$\smallmat{\phantom{-}11}{-22}{-22}{\phantom{-}11\cdot4}$, and
$\smallmat{\phantom{-}3}{-12}{-12}{\phantom{-}12\cdot4}$.
For the four antisymmetric Borcherds products in this example, the
involution conditions (section~\ref{sectionSymm}) say that the second
family of indices,
$n\smallmat{\phantom{-}1}{-2}{-2}{\phantom{-}4}$ for $n\le11$,
indexes Fourier coefficients that are zero.
The involution conditions also say that checking the Fourier
coefficients having indices
$4n\smallmat{\phantom{-}1}{-1}{-1}{\phantom{-}1}$ with~$n\le3$
subsumes checking the Fourier coefficients having indices
$n\smallmat{\phantom{-}1}{-4}{-4}{\phantom{-}16}$ with~$n\le3$.
Thus we may expand only to the third Jacobi coefficient, and this
requires only an expansion of $\psi|V_2$.
This method of checking that the third family of Fourier coefficients
vanishes by checking the first family applies to symmetric forms as well.
For the one symmetric Borcherds product in this example, the second
family of indices shows that we need expansions of~$\psi|V_{10}$ that
subsume $11$ Fourier--Jacobi coefficients.
Checking the symmetric product succeeds for $\nextra=90$ and hence initial
$\Jcusp{46}8$ expansions to~$q^{95}$.  Because $95>23$, this
computation could cite Lemma~\ref{newlemmafour} and 
Theorem~\ref{divboundcor} to skip confirming the truncations~$\tilde\psi$, although things 
were not done in that order.  
This computation again reports the $14$-dimensional subspace of
$\Jwh04$ that is divisible by the baby theta block, and now we
know that this is correct because the divisibility is guaranteed.

The algorithm's process of generating the rest of the Borcherds products
of weight~$46$ and level~$4$ is summarized in Figure~\ref{tbtable3}.
The pairs $(c,t)$ where basic theta blocks relevant to this weight and level
could exist satisfy the discrete quadrilateral bounds from the end
of section~\ref{sectionFind}, $c\ge1$ and
$\max\lset0,46/12-c\rset\le t\le(46-4c)/12$.
Because all the Borcherds products found by the algorithm here are
cuspidal, Figure~\ref{tbtable3} doesn't mention the cuspidality after
each basic theta block.

As this example shows, a high value of~$\nextra$ can be required to
determine whether the Borcherds products found by the algorithm are
cuspidal.  Generally, a high weight~$k$ and/or square factors in the
level~$N$ drive up the necessary $\nextra$.
We have methods to predict the needed~$\nextra$ accurately,
by tracking the leading and trailing exponents of~$\psi$ under the
infinite series Borcherds product formula in Theorem~\ref{BPthm}.
Especially, these methods produced the value $\nextra=90$ in the
penultimate paragraph.

\begin{figure}[htb]
$$
\begin{array}{|c||l|l|l|l|}
\hline
\begin{aligned}
\rule{0pt}{1.0em}
&\phi\in\Jcusp{46}{c\cdot4}\\
&q^{c+t}\,\|\,\phi
\end{aligned}
& \qquad t=0 & \qquad t=1 & \quad t=2 & \quad t=3 \\
\hline\hline
\rule{0pt}{3.4em}
\BPc=1 & \multicolumn{3}{c}{}\vline &
\begin{aligned}
&(1,0)\to5\mathcal S\\
&\quad1\text{sym}\,4\text{anti}\\
&2^2:\\
&\quad i=\underset s3\,\underset a3\,\underset a2\,\underset
  a2\,\underset a2\\
\end{aligned}
\\
\cline{1-1}\cline{4-5}
\rule{0pt}{2.5em}
\BPc=2 & \multicolumn{2}{c}{}\vline &
\begin{aligned}
&(0,1)\to4\mathcal S\\
&1^{-1} 2^2 3^1:\\
&\quad i=3\,2\,2\,2\\
\end{aligned}
&
\\
\cline{1-1}\cline{3-5}
\rule{0pt}{1.9em}
\BPc=3 & &
\begin{aligned}
&(0,1)\to1\mathcal S\\
&1^{-2} 2^2 3^2:i=2\\
\end{aligned}
&
\\
\cline{1-4}
\rule{0pt}{5.75em}
\BPc=4 &
\begin{aligned}
&(1,5)\to4\mathcal S\\
&1^{-8} 2^{10}:i=1\\
&1^{-3} 2^2 3^3:i=1\\
&1^{-4} 2^5 4^1:i=2\\
&4^2:i=2\\
&1^{-1} 2^2 5^1:\emptyset\\
&1^1 2^1 3^{-1} 6^1:\emptyset\\
\end{aligned}
&
\begin{aligned}
&(2,0)\to2\mathcal S\\
&1^8 2^6:i=2\\
&1^{12} 2^1 4^1:i=2\\
&\phantom{1^{-4} 2^5 4^1}\\
&\phantom{4^2}\\
&\phantom{1^{-1} 2^2 5^1}\\
&\phantom{1^1 2^1 3^{-1} 6^1}\\
\end{aligned}
&
\begin{aligned}
&(1,0)\to4\mathcal S\\
&1^{24} 2^2:\\
&\quad i=3\,2\,2\,2\\
&\phantom{1^{-4} 2^5 4^1}\\
&\phantom{4^2}\\
&\phantom{1^{-1} 2^2 5^1}\\
&\phantom{1^1 2^1 3^{-1} 6^1}\\
\end{aligned}\\
\cline{1-4}
\rule{0pt}{2.7em}
\BPc=5 &
\begin{aligned}
&(2,0)\to2\mathcal S\\
&1^7 2^6 3^1:i=1\\
&1^{11} 2^1 3^1 4^1:i=2\\
\end{aligned}
&
\begin{aligned}
&(1,0)\to1\mathcal S\\
&1^{23} 2^2 3^1:i=2\\
&\phantom{1^{11} 2^1 3^1 4^1}\\
\end{aligned}
&
\\
\cline{1-4}
\rule{0pt}{1.9em}
\BPc=6 &
\begin{aligned}
&(1,0)\to1\mathcal S\\
&1^{22} 2^2 3^2:i=1\\
\end{aligned}
&
\\
\cline{1-3}
\rule{0pt}{2.7em}
\BPc=7 &
\begin{aligned}
&(2,0)\to2\mathcal S\\
&1^{32} 2^6:i=1\\
&1^{36} 2^1 4^1:i=2\\
\end{aligned}
&
\begin{aligned}
&(1,0)\to1\mathcal S\\
&1^{48} 2^2:i=2\\
&\phantom{1^{36} 2^1 4^1}\\
\end{aligned}\\
\cline{1-3}
\rule{0pt}{1.9em}
\BPc=8 &
\begin{aligned}
&(1,0)\to1\mathcal S\\
&1^{47} 2^2 3^1:i=1\\
\end{aligned}
&
\\
\cline{1-3}
\BPc=9 &
\\
\cline{1-2}
\rule{0pt}{1.9em}
\BPc=10 &
\begin{aligned}
&(1,0)\to1\mathcal S\\
&1^{72} 2^2:i=1\\
\end{aligned}
\\
\cline{1-2}
\BPc=11 &
\\
\cline{1-2}
\end{array}
$$
\caption{Basic theta blocks and cusp Borcherds products: weight~$46$,
  level~$4$\label{tbtable3}}
\end{figure}

To find all Borcherds product paramodular cusp forms with
specified $(k,N,c,t)$, our program is essentially automated, with
various features possible to activate or not.
To determine all possible pairs $(c,t)$ for a given $(k,N)$ still
requires informed human decisions and other programs.

\bibliographystyle{plain}
\bibliography{allbp}

\end{document}